\documentclass[preprint,1p]{elsarticle}

\makeatletter
 \def\ps@pprintTitle{%
 	\let\@oddhead\@empty
 	\let\@evenhead\@empty
 	\def\@oddfoot{\footnotesize\itshape
 		{} \hfill\today}%
 	\let\@evenfoot\@oddfoot
 }
\makeatother
\usepackage{lipsum} 
\ExplSyntaxOn
\NewDocumentEnvironment{multiequation}{b}
 {
  \vantiempham:n { #1 }
 }
 {}

\seq_new:N \l__vantiempham_md_seq

\cs_new_protected:Nn \vantiempham:n
 {
  \seq_set_split:Nnn \l__vantiempham_md_seq { \\ } { #1 }
  $$
  \seq_map_function:NN \l__vantiempham_md_seq \__vantiempham_md_item:n
  $$
 }
\cs_new_protected:Nn \__vantiempham_md_item:n
 {
  \begin{minipage}{\dim_eval:n {\displaywidth/(\seq_count:N \l__vantiempham_md_seq)} }
  \begin{equation}
  \cs_set_eq:Nc \label { ltx@label }
  #1
  \vphantom{\seq_use:Nn \l__vantiempham_md_seq {}}
  \end{equation}
  \end{minipage}
 }

\ExplSyntaxOff  
\usepackage[unicode]{hyperref}
\usepackage{enumitem}
\makeatletter
\providecommand*{\cupdot}{%
  \mathbin{%
    \mathpalette\@cupdot{}%
  }%
}
\newcommand*{\@cupdot}[2]{%
  \ooalign{%
    $\m@th#1\cup$\cr
    \hidewidth$\m@th#1\cdot$\hidewidth
  }%
}
\makeatother
\usepackage{circuitikz}

\usepackage{tikz-cd}
\usepackage{multirow}
\usepackage{url}
\usepackage{pst-node}
\usepackage{comment}

\usepackage[makeroom]{cancel}
 
\usepackage{soul}
\usepackage{tikz}

\usepackage{latexsym}
\usepackage{indentfirst}
\usepackage{amsxtra}
\usepackage{amssymb}
\usepackage{amsthm}
\usepackage{amsmath}
\usepackage{mathrsfs} 

\usepackage{xcolor}
\usepackage{color}

\usepackage{amsfonts}

\usepackage[capitalise]{cleveref}

\newtheorem{theor}{Theorem}[section]
\newtheorem{prop}[theor]{Proposition}
\newtheorem{lemma}[theor]{Lemma}
\newtheorem{cor}[theor]{Corollary}
\theoremstyle{definition}               
\newtheorem{defin}[theor]{Definition}
\newtheorem{ex}[theor]{Example}
\newtheorem{exs}[theor]{Examples}
\newtheorem{rem}[theor]{Remark}
\newtheorem{rems}[theor]{Remarks}

\newtheorem*{alg*}{Algorithm}

\DeclareMathOperator{\Sym}{Sym}
\DeclareMathOperator{\Aff}{Aff}
\DeclareMathOperator{\Aut}{Aut}
\DeclareMathOperator{\End}{End}
\DeclareMathOperator{\Map}{Map}

\DeclareMathOperator{\id}{id}

\DeclareMathOperator{\Ret}{Ret}
\DeclareMathOperator{\LMlt}{LMlt}
\DeclareMathOperator{\RMlt}{RMlt}

\newcommand{\lambdaa}[2]{\lambda_{#1}{#2}}

\newcommand{\rhoo}[2]{\rho_{#1}{#2}}
\newcommand{\alphaa}[3]{\alpha^{#1}_{#2}{#3}}
\newcommand{\betaa}[3]{\beta^{#1}_{#2}{#3}}

\newcommand{\phii}[2]{\phi_{#1,#2}}

\newcommand{\X}[1]{X_{#1}}
\newcommand{\K}{\mathcal{K}}
\newcommand{\Kinj}{\mathcal{K}_{\text{inj}}}
\newcommand{\Kbij}{\mathcal{K}_{\text{bij}}}

\newcommand{\Z}{\mathbb{Z}}

\usepackage{lastpage}

\makeatletter
\def \@evenhead {\thepage\ of \pageref{LastPage} \hfil \slshape \leftmark } 
\def \@oddhead {{\slshape \rightmark }\hfil \thepage\ of \pageref{LastPage}} 
\makeatother

\begin{document}

\let\today\relax 

\begin{frontmatter}

 	\title{\LARGE{Reflections to set-theoretic solutions of the Yang-Baxter equation} \tnoteref{mytitlenote}}

\author {Andrea~ALBANO}
	
	 \author {Marzia~MAZZOTTA}
	\author{Paola~STEFANELLI}
	\address{Dipartimento di Matematica e Fisica ``Ennio De Giorgi"\\
		Universit\`{a} del Salento\\
		Via Provinciale Lecce-Arnesano \\
		73100 Lecce (Italy)\\}

	\begin{abstract}
    The main aim of this paper is to determine \emph{reflections} to bijective and non-degenerate solutions of the Yang-Baxter equation, by exploring their connections with their derived solutions.
    This is motivated by a recent description of left non-degenerate solutions in terms of a family of automorphisms of their associated left rack. 
    In some cases, we show that the study of reflections for bijective and non-degenerate solutions can be reduced to those of derived type. Moreover, we extend some results obtained in the literature for reflections of involutive non-degenerate solutions to more arbitrary solutions.
    Besides, we provide ways for defining reflections for solutions obtained by employing some classical construction techniques of solutions. Finally, we gather some numerical data on reflections for bijective non-degenerate solutions associated with skew braces of small order.
\end{abstract} 
	

\end{frontmatter}

\section*{Introduction}

The \emph{Yang-Baxter equation} was first introduced by Yang \cite{Ya67} to study a gas of spin with a $\delta$ interaction and, independently, by Baxter \cite{Ba72} to investigate a two-dimensional solvable statistical model. Much of the attention has been directed to the so-called \emph{quantum} Yang-Baxter equation, which is fundamental in mathematical physics since it occupies a central role in the theory of quantum integrable systems \cite{Ji90}. The search for solutions to this equation has attracted numerous authors, and it is still an open problem. Due to the complexity of the solutions, Drinfel'd \cite{Dr92} suggested studying a simplified class. Into the specific, if $X$ is a set and
$r:X\times X\to X\times X$ is a map, the pair $(X, r)$ is said to be a \textit{set-theoretic solution of the Yang-Baxter equation}, briefly a \emph{solution}, if the identity
\begin{align}\tag{YBE}
\left(r\times\id_X\right)
\left(\id_X\times r\right)
\left(r\times\id_X\right)
= 
\left(\id_X\times r\right)
\left(r\times\id_X\right)
\left(\id_X\times r\right)
\end{align}
holds.  Writing $r\left(x, y\right) = \left(\lambda_{x}\left(y\right), \rho_{y}\left(x\right)\right)$, where $\lambda_{x}$ and $\rho_{y}$ are maps from $X$ into itself, for all $x,y\in X$, then
$r$ is a solution if and only if
    \begin{align}
     &\label{first} \lambda_x\lambda_y(z)=\lambda_{\lambda_x\left(y\right)}\lambda_{\rho_y\left(x\right)}\left(z\right)\tag{Y1}\\
    &  \label{second}\lambda_{\rho_{\lambda_y\left(z\right)}\left(x\right)}\rho_z\left(y\right)=\rho_{\lambda_{\rho_y\left(x\right)}\left(z\right)}\lambda_x\left(y\right)\tag{Y2}\\
      &\label{third}\rho_z\rho_y(x)=\rho_{\rho_z\left(y\right)}\rho_{\lambda_y\left(z\right)}\left(x\right)\tag{Y3}
  \end{align} 
  for all $x,y,z \in X$. With this notation, 
a solution is \textit{bijective} if $r$ is  a bijective map; \textit{involutive} if $r^2=\id_{X\times X}$; \emph{idempotent} if $r^2=r$; \textit{left non-degenerate} if $\lambda_x \in \Sym_X$, for all $x \in X$; \textit{right non-degenerate} if $\rho_x \in \Sym_X$, for all $x \in X$, and \emph{non-degenerate} if it is both left and right non-degenerate. In the beginning, many authors mainly focused on the class of bijective non-degenerate solutions. In the years that followed, the attention has been shifted to more arbitrary classes of solutions and a wide range of papers has been published in the literature.

Similarly to the YBE, the \emph{reflection equation} is an important tool in the theory of quantum groups and integrable systems. This equation was first studied by Cherednik \cite{Ch84} to encode the reflection on the boundary of particles in quantum field theory and by Sklyanin \cite{Sk88} to prove the integrability of quantum models with boundaries. The set-theoretic version of this equation was formulated by Caudrelier and Zhang in \cite{CaZh14}, who also provided the first examples of solutions. It turns out that the set-theoretic reflection equation jointly with the YBE ensures the factorization property of the interactions of $N$-soliton solutions on the half-line. For an explanatory illustration, see Figs. 1-2 in \cite{LeVe22}. \\
If $(X,r)$ is an arbitrary solution, a map $\kappa:X\to X$ is a \textit{set-theoretic solution of the reflection equation}, or briefly a \emph{reflection}, for $(X,r)$ if it holds
	\begin{align}\label{RE}\tag{RE}
	r\left(\id_{X}\times \kappa\right)r\left(\id_{X}\times \kappa\right)
	= \left(\id_{X}\times \kappa\right)r\left(\id_{X}\times \kappa\right)r.
	\end{align}
 One can easily see that \eqref{RE} is equivalent to requiring that, for all $x, y \in X$,
 \begin{align}\label{t1}
     &\lambda_{\lambda_x(y)} \kappa\rho_y\left(x\right)=\lambda_{\lambda_x(\kappa(y))} \kappa\rho_{\kappa(y)}\left(x\right), \tag{T}\\
     &\rho_{\kappa\rho_{\kappa(y)}(x)} \lambda_x{\kappa(y)}=\kappa\rho_{\kappa\rho_y(x)} \lambda_x\left(y\right). \label{q1}\tag{Q}
 \end{align}
The interplay between solutions to the YBE and reflections was deepened by Caudrelier, Cramp\'e, and Zhang in \cite{CaCrZh13}. In particular, they provided a method for constructing all the reflection maps associated with the quadrirational Yang–Baxter maps.
In the following years, Smoktunowicz, Vendramin, and Weston \cite{SmVeWe20} proposed a first more systematic approach to the set-theoretic case that makes use of ring-theoretic methods, and more generally methods coming from
the theory of braces, to produce families of reflections. The reflections they provide belong to the class of what we call \emph{$\lambda$-centralizing maps}, namely, maps $\kappa$ such that $\lambda_x \kappa=\kappa \lambda_x$, for all $x \in X$.  Furthermore, De Commer \cite{Decommer19} introduces a notion of braided action of a group with braiding and shows that it provides reflections.
Besides, Doikou and Smoktunowicz \cite{DoSm21} investigate connections between set-theoretic Yang–Baxter and reflection equations and quantum integrable systems. More recently, Lebed and Vendramin \cite{LeVe22} have constructed a family of solutions $r^{(\kappa)}$ indexed by a reflection $\kappa$ for a right non-degenerate solution $(X, r)$, including the solution $r$ itself and its derived solution. They also show that if $(X, r)$ is an involutive left non-degenerate solution, a map $\kappa$ is a reflection if and only if \eqref{t1} holds; similarly, if $(X, r)$ is right non-degenerate, $\kappa$ is a reflection if and only if \eqref{q1} holds. Thus, $\lambda$-centralizing maps are reflections for involutive left-non degenerate solutions, while maps $\kappa$ satisfying $\rho_{\kappa(x)}=\rho_x$, for all $x \in X$, that we call \emph{$\rho$-invariant},
are reflections for the involutive right non-degenerate ones. We mention other interesting works on the reflection equation as \cite{KuSa93, AzKu94, DoMa03, Kat19x, KuOk19, Do24}.

The main aim of this paper is to study reflections for non-degenerate solutions employing the left and the right shelf structures associated with them. We briefly recall that a \emph{left shelf} is a set $X$ endowed with a left self-distributive binary operation $\triangleright$, namely, the identity $x\triangleright\left(y\triangleright z\right) = \left(x\triangleright y\right)\triangleright\left(x\triangleright z\right)$ is satisfied, for all $x,y,z\in X$. Equivalently, for all $x \in X$, the left multiplication $L_x:X\to X$ defined by $L_x\left(y\right) = x\triangleright y$, for all $y\in X$, is a shelf homomorphism. In particular, if $\left(X, r\right)$ is a left non-degenerate solution, set $x \triangleright_r y:= \lambda_x\rho_{\lambda_y^{-1}(x)}(y)$, for all $x,y\in X$, one obtain that $\left(X, \triangleright_r\right)$ is a left shelf that gives rise to a solution  $\left(X,r_\triangleright\right)$, with $r_\triangleright\left(x,y\right) = \left(y, y\triangleright_r x\right)$, for all $x,y\in X$, called the \emph{left derived solution of $r$}. Analogously, one can define the structure $\left(X, \triangleleft\right)$ of \emph{right shelf}, introduce the right multiplications $R_x: X \to X$ given by $R_x(y)= y \triangleleft x$, for all $y \in X$, and the link with right non-degenerate solutions. For our purposes, it will also be useful to introduce classical tools of this theory, such as the \textit{left multiplication group} of a left shelf $(X, \triangleright)$ given by $\LMlt(X,\, \triangleright)=\langle L_x \, \mid \ x \in X\rangle$. For more details on this topic, see \cite{FeSaFa04, FeRo92,  Le18, LeVe19}.

Our study is motivated by a recent description of left non-degenerate solutions in terms of left shelves structures provided in \cite[Theorem 2.15]{DoRySt24x}. 
Indeed, any left non-degenerate solution can be written in terms of a family of automorphisms of the left shelf associated with it, i.e., the \emph{Drinfel'd twist}, and of the shelf operation itself. 
This naturally suggests splitting the study of reflections into their behaviour with respect to the maps $\lambda_x\in\Aut\left(X, \triangleright_r\right)$ and left multiplications $L_x$. Hence, reflections further motivate the investigation of the endomorphism monoid of shelves.

We first focus on reflections for solutions that are only left or right non-degenerate since, in our analysis, they turn out to be different from each other. Following some existing approaches mentioned before, we pay particular attention to $\lambda$-centralizing and $\rho$-invariant reflections and characterize them in the left and right non-degenerate cases in \cref{prop:refl-lnd} and in \cref{prop:refl-rnd}, respectively. Besides, we reformulate the results achieved in terms of the respective multiplication groups of their associated shelf structure (see \cref{bij_refl:bij_nondeg_sol} and \cref{cor_right-bij}). Then, we deal with the non-degenerate bijective case, by combining the observations made separately for the left and right cases, and developing a theory for such solutions. In this context, in Appendix \ref{sec:appendix}, we gather some numerical data on reflections for bijective non-degenerate solutions associated with skew braces, dividing them by classes of belonging.

The last section is devoted to constructing reflections starting from known ones. More specifically, considering construction techniques for arbitrary solutions, as the \emph{matched product} \cite{CCoSt20} and the \emph{strong semilattice of solutions} \cite{CCoSt20x-2}, we define reflections for the new solution obtained, starting from reflections on each solution involved in the construction itself. 
Moreover, we focus on the \emph{twisted union}, a method for obtaining involutive non-degenerate solutions provided in \cite{ESS99}. In particular, we first extend this technique to arbitrary solutions and then we show how to construct a reflection for new solutions determined.

\bigskip


\section{Preliminaries}
In this section, for the ease of the reader, we provide the necessary background for the rest of the paper.

\subsection{Basics on the reflection equation}

\medskip

\begin{defin}\cite[Definition 1.1]{CaCrZh13}
	Let $(X,r)$ be a solution. A map $\kappa:X\to X$ is a \textit{set-theoretic solution of the reflection equation}, or briefly a \emph{reflection}, for $(X,r)$ if it holds
	\begin{align*}
	r\left(\id_{X}\times \kappa\right)r\left(\id_{X}\times \kappa\right)
	= \left(\id_{X}\times \kappa\right)r\left(\id_{X}\times \kappa\right)r.
	\end{align*}
 Moreover, we say that the reflection $\kappa$ is \emph{involutive} if $\kappa^{2} = \id_{X}$.
\end{defin}
 We denote with $\K(X, r)$ the set of all reflections for a solution $(X, r)$. Moreover, $\Kinj(X,r)$ and $\Kbij(X,r)$ will denote the set of all \emph{injective} and \emph{bijective} reflections for $(X, r)$, respectively. If the solution is clear from the context, we will briefly write $\K(X)$.
Note that $\K(X)$ is non-empty since the identity map $\id_X$ is always a reflection for all solutions defined on the set $X$.

\begin{rems}\label{iso:refl} \hspace{1mm}
If $(X,r)$ is a solution, let us observe some easy but essential properties:
\begin{enumerate}
\item If $(Y, s)$ is equivalent to $(X,r)$ via a bijection $\varphi: X \to Y$ (see \cite{ESS99}), namely, $(\varphi \times \varphi)r=s(\varphi \times \varphi)$, and $\kappa \in \K(X, r)$, then $\kappa_{\varphi}:=\varphi\kappa\varphi^{-1} \in \K(Y,s)$. 
\item    If $(X,r)$ is bijective and $\kappa \in \Kbij(X,r)$, then $\kappa^{-1} \in \Kbij(X,r^{-1})$.
\end{enumerate}
\end{rems}

\smallskip

In the following, we provide some examples of reflections of solutions.

\begin{exs}\label{exs_riflessioni} 
Assume that $X$ is a set.
\begin{enumerate}
    \item   Any map $\kappa: X \to X$ is a reflection for the solution $\left(X, \id_X\right)$.
    \item Let $c\in X$ and let $\left(X, r\right)$ be the idempotent solution given by $r\left(x, y\right) = \left(c, c\right)$, for all $x, y \in X$. 
    Then, it is easy to check that any map $\kappa:X\to X$ such that $\kappa\left(c\right) = c$ is a reflection for $\left(X, r\right)$.
    \item Let $\left(X, r\right)$ be the idempotent left non-degenerate solution with $r\left(x, y\right) = \left(y, y\right)$, for all $x, y \in X$. Then, a map $\kappa:X\to X$ is a reflection for $\left(X, r\right)$ if and only if $\kappa^2 = \kappa$.
       \item The unique reflection for the idempotent right non-degenerate solution $\left(X, s\right)$ given by $s\left(x, y\right) = \left(x, x\right)$, for all $x, y \in X$, is the identity map $\id_X$.
\end{enumerate}
\end{exs}

Note that even though the solutions $(X, r)$ and $(X, s)$ in \cref{exs_riflessioni}-$3.\&4.$ are such that $r=\tau s \tau$, with $\tau$ the flip map, their reflections do not share any relevant property.

\begin{ex}\cite[Example 1.4]{SmVeWe20}\label{ex:fg}
Let $X$ be a set, $f,g \in \Sym_X$ such that $fg=gf$, and consider the bijective non-degenerate solution $r(x,y)=(f(y), g(x))$ belonging to the class of Lyubashenko's solutions \cite{Dr92}.
A map $\kappa: X \to X$ is a reflection for $(X, r)$ if and only if $\kappa$ commutes with $fg$. In particular, if $(X, r)$ involutive, i.e., $g=f^{-1}$, any map $\kappa: X \to X$ is a reflection for $(X, r)$.
\end{ex}
\begin{ex}\cite[Example 1.3]{SmVeWe20}\label{ex_123}
Let $X = \{1, 2, 3\}$ and $r(x, y) = \left(f_x\left(y\right), f_y\left(x\right)\right)$ the involutive non-degenerate solution given by $f_1=f_2=\id_X$ and
$f_3=(1 \ 2)$. Then, there exist five reflection maps. To write them, we denote the image of $\kappa$ as the string $\kappa(1)\kappa(2)\kappa(3)$.  They are given by
\begin{align*}
 \kappa_1=123=\id_X, \quad\kappa_2=113, \quad  \kappa_3=213, \quad \kappa_4=223, \quad \kappa_5=333.
\end{align*}
Note that $\kappa_j$ is involutive if and only if $j \in \{1, 3\}$.
\end{ex}

\smallskip

In \cite[Lemma 1.7]{SmVeWe20}, it is introduced the following notation to study the reflections of a solution $(X, r)$. 
\begin{lemma}\label{lemma_rifl_qt}
Let $(X, r)$ be a solution and let $\kappa: X \to X$ be a map. Set
    \begin{align*}
   t(x,y):=\lambda_{\lambda_x(y)} \kappa\rho_y\left(x\right) \quad \text{and} \quad   q(x,y):=\rho_{\kappa\rho_y(x)} \lambda_x\left(y\right),
\end{align*}
for all $x,y\in X$. Then, $\kappa \in \K(X, r)$ if and only if the following hold
    \begin{align}
   t(x,\kappa(y))=t(x,y) \label{t} \tag{T}\\
   q(x, \kappa(y))=\kappa(q(x,y)) \label{q} \tag{Q}
\end{align}
for all $x,y\in X$.
\end{lemma}

\medskip

In  \cite[Theorem 3.2]{LeVe22}, it is shown that for involutive non-degenerate solutions, it is sufficient to check only one of the relations \eqref{t} and \eqref{q}.

\begin{prop}\label{reflec_invol}\hspace{1mm}
\begin{enumerate}
    \item Let $(X, r)$ be an involutive left non-degenerate solution and $\kappa: X\to X$ a map. Then, $\kappa \in \K(X, r)$ if and only if \eqref{t} holds.
    \item Let $(X, r)$ be an involutive right non-degenerate solution and $\kappa: X\to X$ a map. Then, $\kappa \in \K(X, r)$ if and only if \eqref{q} holds.
\end{enumerate}
\end{prop}

\smallskip

The authors in \cite[Theorem 1.9]{SmVeWe20} give sufficient conditions for a map $\kappa$ to be a reflection for an involutive left non-degenerate solution. To recall such a result, let us introduce the following notion.

\begin{defin}
Let $(X,r)$ be a solution. Then, a map $\kappa:X \to X$ is said to be \emph{$\lambda$-centralizing} if \begin{align}\label{lambda-centr}\tag{L}
    \forall\, x \in X \quad \lambda_x\kappa  = \kappa\lambda_x.
\end{align}
Similarly one can define the notion of \emph{$\rho$-centralizing} map.
\end{defin}

\noindent Notice that this notion was first given in \cite{SmVeWe20} in terms of $\mathcal{G}(X, r)$-equivariance, where $\mathcal{G}(X, r)$ is the involutive Yang-Baxter group generated by the maps $\lambda_x$.

\begin{prop} \label{left}
    Let $(X, r)$ be an involutive left non-degenerate solution. Then, any $\lambda$-centralizing map is a reflection for $(X, r)$.
\end{prop}

We underline that, in general, not all the reflections for an involutive left non-degenerate solution are $\lambda$-centralizing (see, for instance, $\kappa_2$ in \cref{ex_123}). 
Regarding the right non-degenerate case, we recall the following result contained in
\cite[Theorem 3.2]{LeVe22}.
\begin{prop}\label{right}
Let $(X, r)$ be an involutive right non-degenerate solution and let us assume that $\kappa: X \to X$ is a map such that 
\begin{align}\label{rho_inv}\tag{R}
 \forall\, x \in X \quad   \rho_x=\rho_{\kappa(x)}, 
\end{align}
then $\kappa \in \K(X, r)$.
\end{prop}
\noindent Also in this case, this is only a sufficient condition, since the reflection $\kappa_5$ in \cref{ex_123} does not satisfy \eqref{rho_inv}. To study such a kind of maps, we introduce the following terminology:
\begin{defin}
   Let $(X, r)$ be a solution. Then, we call \emph{$\rho$-invariant} any map $\kappa: X \to X$ satisfying the property \eqref{rho_inv}.
   Similarly, one defines the notion of $\lambda$\emph{-invariance}.
\end{defin}

\begin{rem}
    The properties \eqref{lambda-centr} and \eqref{rho_inv} are invariant under equivalence.
    Indeed, if $(X, r)$ and $(Y, s)$ are equivalent solutions via a bijection $\varphi:X \to Y$ and $\kappa\in \K(X, r)$, then the following hold:
    \begin{enumerate}
        \item $\kappa$ is $\lambda$-centralizing if and only if $\kappa_{\varphi}$ is $\lambda$-centralizing;
        \item $\kappa$ is $\rho$-invariant if and only if $\kappa_{\varphi}$ is $\rho$-invariant.
    \end{enumerate}
\end{rem}

\smallskip

More generally, the following can be proven for a $\lambda$-invariant or $\rho$-invariant reflection.
\begin{lemma}\label{rem:obs}
Let $(X, r)$ be a solution and $\kappa: X \to X$ a map.
    \begin{enumerate}
        \item If $\left(X, r\right)$ is left non-degenerate and $\kappa$ is $\rho$-invariant for $(X, r)$, then:
\begin{align*}
            \kappa\ \text{is $\lambda$-invariant for $(X, r)$}
           \iff
           \forall\,x,y\in X\quad\lambda_{\lambda_x\kappa\left(y\right)} = \lambda_{\lambda_x\left(y\right)}.
        \end{align*}
        \item If  $\left(X, r\right)$ is right non-degenerate and $\kappa$ is $\lambda$-invariant,  for every $x\in X$, then:
\begin{align*}
            \kappa\ \text{is $\rho$-invariant for $(X, r)$}
           \iff
           \forall\,x,y\in X\quad\rho_{\rho_x\kappa\left(y\right)} = \rho_{\rho_x\left(y\right)}.
        \end{align*}
    \end{enumerate}
\end{lemma}

The following are sufficient conditions for obtaining reflections of an involutive left (resp. right) non-degenerate solution $(X, r)$ belonging to the class of $\rho$-invariant maps.

\begin{prop}\label{prop:lambdakappa}
    Let $(X, r)$ be an involutive solution and let $\kappa:X\to X$ be a map.
    \begin{enumerate}
        \item  If $(X, r)$ is left non-degenerate and  it holds that $\lambda_{\lambda_x\kappa\left(y\right)} = \lambda_{\lambda_x\left(y\right)}$, for all $x,y\in X$, then $\kappa\in\K\left(X, r\right)$ and is $\rho$-invariant.
    \item If $(X, r)$ is right non-degenerate and it holds that $\rho_{\rho_x\kappa\left(y\right)} = \rho_{\rho_x\left(y\right)}$, for all $x, y \in X$, then $\kappa\in\K\left(X, r\right)$ and is $\rho$-invariant.
    \end{enumerate}
\begin{proof}
The first statement follows by \cref{reflec_invol}-$1.$. Moreover, since $(X, r)$ is involutive, by the assumptions, $\kappa$ is $\rho$-invariant. For the second one, it is enough to observe that $\lambda_x = \lambda_{\kappa\left(x\right)}$, for every $x\in X$. Thus,  by \cref{reflec_invol}-$2.$ and \cref{rem:obs}, $\kappa\in\K\left(X, r\right)$ and $\kappa$ is $\rho$-invariant.
\end{proof}
\end{prop}

\medskip

\subsection{Some basics on shelf structures}
Let us start this section by recalling some basics on shelves and racks. For further details, we refer to \cite{FeSaFa04, FeRo92,  Le18, LeVe19}.

\smallskip

A \emph{left shelf} $(X, \triangleright)$ is a set $X$ equipped with a left self-distributive binary operation $\triangleright$ such that
\begin{equation}\label{self_distri}
    \forall\, x,y,z\in X\quad x \triangleright (y \triangleright z)= (x \triangleright y) \triangleright (x \triangleright z)\,. 
\end{equation} 

A \emph{shelf homomorphism} between two shelves is defined as any map preserving the shelf operations on each underlying set.
Hereafter, we will denote with $\End(X,\triangleright)$ and $\Aut(X,\triangleright)$ the endomorphism monoid and the automorphism group of a given shelf $(X,\triangleright)$, respectively.
If we denote with $L_x:X \to X, y \mapsto x \triangleright y$ the operator of left multiplication by $x$, for all $x \in X$, then equation \eqref{self_distri} is equivalent to requiring that $L_x$ is a shelf endomorphism.

A left shelf $(X, \triangleright)$ is a \emph{left rack} if all the maps $L_x$ are bijective. Moreover, a rack $(X, \triangleright)$ is a \emph{left quandle} if, in addition, $L_x(x) = x$, for all $x \in X$.

In a similar way, it is possible to define a \emph{right shelf} $(X, \triangleleft)$, and analogously a \emph{right rack} as a right shelf $(X, \triangleleft)$ having all the right multiplications $R_x$ bijective.
\smallskip

Below, we give some useful examples of left shelves and left racks, some of which are well-known in literature. Clearly, one can derive their right version.

\begin{exs}\label{ex_shelf}\hspace{1mm}
\begin{enumerate}
    \item If $X$ is a set and $a \in X$, set $x \triangleright y:=a$, for all $x, y \in X$, then $(X, \triangleright)$ is a left shelf.   
       \item Let $X$ be a set, $f: X \to X$ an idempotent map, and define $x \triangleright y:=f(x)$, for all $x, y \in X$, then $(X, \triangleright)$ is a left shelf. 
    \item Let us consider $(\mathbb{Z}_6, +)$ and define $x \triangleright y := 2x+2y$ $\left(\text{mod} \, \,6\right)$, for all $x, y \in \mathbb{Z}_6$. Then, $\left(\mathbb{Z}_6, \triangleright\right)$ is a left shelf, cf. \cite[Example 2.2]{DoRySt24x}.
\end{enumerate} 
\end{exs}

\begin{exs}\hspace{1mm} 
\begin{enumerate} 
\item (\textit{Permutation Rack}) Let $X$ be a set, $f \in \Sym_X$, and set $x \triangleright y:= f(y)$, for all $x, y \in X$. If $f=\id_X$, then $(X, \triangleright)$ is the \emph{trivial rack}.
    \item (\textit{Cyclic Rack of order n}) Let $n \in \mathbb{N}$ 
    and $X=\mathbb{Z}_n$. Define $x \triangleright y := y+1$ $\left(\text{mod} \, \, n\right)$, for all $x,y \in X$.
    Then, $(X,\triangleright)$ is a left rack which is a quandle if and only if $n=1$. 
    \item (\textit{Conjugation quandle}) Let $G$ be a group and define $x \triangleright y := x^{-1}yx$, for all $x,y \in G$. Then, the pair $(G,\triangleright)$ is a left quandle. More generally, any possible union of conjugacy classes of $G$ inherits the structure of a left quandle.
    \item (\textit{Dihedral quandle}) Let $n\in \mathbb{N}_0$ and $X =\Z_n$. Define a binary operation on $X$ by setting $i \triangleright j := 2i-j$ $\left(\text{mod} \, \, n\right)$, for all $i, j \in X$. Then, $(X, \triangleright)$ is a left quandle. In particular, the case $n = 0$ provides the infinite dihedral quandle.
\end{enumerate}  
\end{exs}

\smallskip

The following results explain the relationship between shelf structures and solutions that will be useful for our purposes.

\begin{prop}\label{YBE_and_shelf}
    Let $(X, \triangleright)$ be a left shelf.
    Then, the map $r_\triangleright:X\times X\to X\times X$ given by 
$r_\triangleright(x, y)=\left(y, y\triangleright x\right)$,
    for all $x, y \in X$, is a left non-degenerate solution on $X$, called the \emph{solution associated to the left shelf $(X, \triangleright)$}. \\Conversely, if $(X, r)$ is a left non-degenerate solution, set
    $$x \triangleright_r
    y:= \lambda_x\rho_{\lambda_y^{-1}(x)}(y),$$
    for all $x, y \in X$, then the structure $(X, \triangleright_r)$ is a left shelf, called the \emph{left shelf associated to the solution $(X, r)$}. The solution associated with $\left(X, \triangleright_r\right)$ is called the \emph{left derived solution} of $(X,r)$.
\end{prop}

The correspondence between right non-degenerate solutions and right shelves is analogous and easily obtained by exchanging the roles of the maps $\lambda$ and $\rho$.
Specifically, if $(X,r)$ is right non-degenerate, then its associated right shelf $(X,\triangleleft_r)$ is defined by setting 
$$
\forall\ x,y \in X \qquad x \triangleleft_r y := \rho_y\lambda_{\rho^{-1}_x(y)}(x).
$$

\medskip

\begin{prop}
     Let $(X, \triangleright)$ be a shelf. Then, the solution $\left(X, r_\triangleright\right)$ is non-degenerate if and only if $(X, \triangleright)$ is a rack.
\end{prop}

\smallskip

According to \cite[Theorem 2.15]{DoRySt24x}, it is possible to describe any left (or right) non-degenerate solution in terms of its associated shelf. We recall such a result in the direction useful for our purposes.

\begin{theor}\label{sol:lnd}\hspace{1mm}
\begin{enumerate}
    \item     Let $(X, r)$ be a left non-degenerate solution and let $(X,\triangleright_r)$ its associated shelf.
    Then,  $\lambda_x\in\Aut\left(X,\triangleright_r\right)$ and $\rho_y\left(x\right) = \lambda^{-1}_{\lambda_x\left(y\right)}\left(\lambda_x\left(y\right)\triangleright_r x\right)$,
    for all $x,y\in X$. 
    \item Let $(X, r)$ be a right non-degenerate solution and let $(X,\triangleleft_r)$ its associated shelf.
    Then,  $\rho_x\in\Aut\left(X,\triangleleft_r\right)$ and $\lambda_x\left(y\right) = \rho^{-1}_{\rho_y\left(x\right)}\left(y \triangleleft_r \rho_y(x)\right)$,
    for all $x,y\in X$. 
\end{enumerate}
\end{theor}
Let us observe that, in particular, any left non-degenerate involutive solution on a set $X$ is the solution associated with the trivial rack.

\medskip

It turns out to be useful to introduce the following group associated with a left rack structure. Similarly, one can define it for a right rack.
\begin{defin}
Let $(X,\triangleright)$ be a left rack.    
The \emph{left multiplication group} of $(X,\triangleright)$ is the normal subgroup of $\Aut(X)$ defined by
    \begin{align*}
        \LMlt(X,\triangleright) := \langle L_x \mid x \in X \rangle\,.
    \end{align*}
\end{defin}
\noindent  This group is an important invariant of a rack and computing it, in general, is a fundamental task in the theory (for example, see \cite{FeRo92, ElMaRe12}).

\medskip

\section{Reflections for left non-degenerate solutions}
In this section, we focus on reflections for left non-degenerate solutions and, consequently, on left shelf structures. Specifically, we describe $\lambda$-centralizing reflections for this class of solutions, i.e., reflections commuting with all the maps $\lambda_x$.
\medskip

\begin{theor}\label{prop:refl-lnd}
    Let $(X,r)$ be a left non-degenerate solution, $(X, \triangleright_r)$ its associated shelf, and $\kappa:X \to X$ a $\lambda$-centralizing map.
    Then, $\kappa \in \K(X,r)$ if and only if, for all $x,y \in X$, the following hold:
    \begin{enumerate}
        \item\label{eq:1} $\kappa L_x = \kappa L_{\kappa\left(x\right)}$,
        \item\label{eq:2} $\kappa L_{L_x(y)}(x) = L_{\kappa L_x(y)}\kappa(x)$\,.
    \end{enumerate}
    \begin{proof}
    Initially, using \cref{sol:lnd}, observe that, by an easy computation, the equation in \ref{eq:1}. is equivalent to \eqref{t}.
   Now, notice that
    \begin{align*}
        \kappa q\left(x,y\right) = \lambda^{-1}_{t\left(x,y\right)} \kappa L_{t\left(x,y\right)} \lambda_x\left(y\right)
        \quad\text{and}\quad
        q\left(x,\kappa\left(y\right)\right) =     \lambda^{-1}_{t\left(x,\kappa\left(y\right)\right)}L_{t\left(x,\kappa\left(y\right)\right)} \lambda_x\kappa\left(y\right),
    \end{align*}
    for all $x,y\in X$.
    Thus, if $\kappa$ is a reflection for $(X, r)$, we get
    \begin{align*}
        \forall\,x,y\in X\quad \kappa q(x,y) = q(x,\kappa(y))&\Leftrightarrow
        \forall\,x,y\in X\quad \kappa L_{t\left(x,y\right)} \lambda_x\left(y\right) = L_{t\left(x,y\right)} \lambda_x\kappa\left(y\right)\\
        &\Leftrightarrow
        \forall\,x,y\in X\quad \kappa L_{\kappa L_{\lambda_x(y)}(x)} \lambda_x\left(y\right) = L_{\kappa L_{\lambda_x(y)}(x)} \lambda_x\kappa\left(y\right) \\
        &\underset{\eqref{eq:1}}{\Leftrightarrow} 
        \forall\,x,y\in X\quad \kappa L_{L_{\lambda_x(y)}(x)} \lambda_x\left(y\right) = L_{\kappa L_{\lambda_x(y)}(x)}\kappa\lambda_x\left(y\right),
    \end{align*}
    i.e., \ref{eq:2}. holds.  Conversely, if \ref{eq:1}. and \ref{eq:2}. hold then, by following the previous chain of equivalences we also obtain \eqref{q}. 
    \end{proof}   
\end{theor}

We remark that the previous result extends \cref{left}, since if $(X,r)$ is an involutive left non-degenerate solution then $(X,\triangleright_r)$ is the trivial rack.
\smallskip

\begin{cor}\label{cor:der-left}
Let $(X,r)$ be a left non-degenerate solution and let us assume that $\kappa$ is a $\lambda$-centralizing map for $(X, r)$. Then, $\kappa \in \K(X, r)$ if and only if $\kappa \in \K(X, r_{\triangleright})$. 
    \begin{proof}
        The claim follows by \cref{YBE_and_shelf} and \cref{prop:refl-lnd}.
    \end{proof}
\end{cor}

Notice that property $2.$ of  \cref{prop:refl-lnd} might suggest that a $\lambda$-centralizing reflection $\kappa$ of a left non-degenerate solution $(X, r)$ is an endomorphism of its associated rack $\left(X, \triangleright_r\right)$. However, in general, this is not the case, as one can verify in the following example.

\begin{ex}
If we consider the shelf $\left(X,\triangleright\right)$  in \cref{ex_shelf}-2. on $X= \{1,\,2,\,3\}$ with $f:X\to X$ given by $f= 122$, then the map $\kappa:X\to X$ defined by $\kappa = 112$ is a reflection that is not a shelf endomorphism of $\left(X,\triangleright\right)$.
\end{ex}

Therefore, one question that arises is whether left multiplications are reflections for a solution associated with a shelf structure.

\begin{prop}
    \label{moltipl_rifl_shelf} Let $(X,\triangleright)$ be a left shelf, $(X,r_{\triangleright})$ its associated solution, and $a\in X$. Then, the following conditions are equivalent:
    \begin{enumerate}
        \item $L_a \in \K(X,r_\triangleright)$;
        \item $L_aL_x = L_aL_{L_a(x)}$, for every $x\in X$;
        \item $ L_{L_a(x)}L_a= L_aL_{L_a(x)}$, for every $x\in X$.
    \end{enumerate}
\begin{proof}
 It directly follows from \cref{prop:refl-lnd} and the self-distributivity property \eqref{self_distri}. 
\end{proof}
\end{prop}

\smallskip

The property \ref{eq:2}. of \cref{prop:refl-lnd} further simplifies if we assume that $L_x$ is bijective, for all $x \in X$, i.e., $(X,\,\triangleright_r)$ is a rack.
According to \cite[Corollary 2.16]{DoRySt24x}, this is the case of a left non-degenerate solution which is also bijective. 
\begin{cor}\label{cor:refl_bilnd}
    Let $(X,r)$ be a bijective left non-degenerate solution and $\kappa:X \to X$ a $\lambda$-centralizing map. Then, $\kappa \in \K(X,r)$ if and only if $\kappa \in \End\left(X, \triangleright_r\right)$ such that
\begin{align}\label{id-cor:refl_bilnd}
  \forall\, x \in X \quad \kappa L_x = \kappa L_{\kappa\left(x\right)}\,.  
\end{align}
In particular, the following hold:
\begin{enumerate}
    \item if $\kappa \in \mathcal{K}(X, r)$ is involutive, then $\kappa$ is an involutive automorphism of $\left(X,\,\triangleright_r\right)$.
    \item $\kappa \in \mathcal{K}(X, r)$ is idempotent, i.e., $\kappa^2=\kappa$, if and only if $\kappa$ is an idempotent of $\End\left(X,\, \triangleright_r\right)$;
\end{enumerate}
\begin{proof}
   The result is an easy consequence of \cref{prop:refl-lnd}.  
   Moreover, $1.$ is obvious.
   Finally, to prove $2.$, it is sufficient to observe that if we assume that $\kappa:X \to X$ is an idempotent endomorphism of $(X,\,\triangleright_r)$, then the following holds, for all $x \in X$:
    \[
        \kappa L_{\kappa(x)} = L_{\kappa^2(x)} \kappa =
        L_{\kappa(x)} \kappa = \kappa L_x \,,
    \]
    thus finishing the proof.
\end{proof}
\end{cor}

Observe that, in general, the converse of \cref{cor:refl_bilnd}-$1.$ is not true (see \cref{red-contrex}).

\smallskip

\begin{rem}
    Reflection maps share an unexpected connection with averaging operators on racks. According to \cite{das24x}, if $(X,\,\triangleright)$ is a rack then a map $A:X \to X$ is a (left) \emph{averaging operator on $(X,\,\triangleright)$} if\,
    $A\left(x\right)\triangleright A\left(y\right) = A\left(A\left(x\right)\triangleright y\right)$, for all $x,y\in X$. Evidently, if $A\in\End(X,\,\triangleright)$, then $A$ is averaging operator on\, $(X,\,\triangleright)$ if and only if $A$ satisfies the identity \eqref{id-cor:refl_bilnd} in \cref{cor:refl_bilnd}.
\end{rem}


\begin{rem}
  Let $\kappa_1, \kappa_2$ be $\lambda$-centralizing reflections of a bijective left non-degenerate solution $(X,r)$.
  If $\kappa_1\kappa_2=\kappa_2\kappa_1$, then $\kappa_1\kappa_2$ is a reflection of $(X, r)$. The converse is not true since in \cref{ex_123} we have that $\kappa_2\kappa_3=\kappa_2$ and $\kappa_3\kappa_2=\kappa_4$.
  Moreover, in general, the commutation assumption can not be relaxed since for a given solution there can exist $\lambda$-centralizing reflections $\kappa$ and $\omega$ whose composition is not again a reflection (see \cref{exs_riflessioni}-3.).
\end{rem}

In the case of a non-degenerate left derived solution, \cref{prop:refl-lnd} allows to characterize all reflections only in terms of the associated rack.

\begin{cor}
    \label{refl_rack_char}
    Let $\left(X, \triangleright\right)$ be a left rack and $(X,r_{\triangleright})$ its associated solution. Then, a $\kappa \in \K(X,r_{\triangleright})$ if and only if $\kappa \in \End\left(X, \triangleright_r\right)$ such that
    \begin{center}
   $\forall\, x \in X \quad \kappa L_x = \kappa L_{\kappa\left(x\right)}.$     
    \end{center}
\end{cor}

\smallskip 

The \cref{refl_rack_char} allows us to easily find examples of reflections of the solutions associated with racks.
\begin{exs}\label{ex_refl_rack} Let $(X,\triangleright)$ be a left rack.
\begin{enumerate} 
    \item Any idempotent endomorphism of $(X,\triangleright)$ lies in $\K(X)$. 
    \item If $(X,\triangleright)$ is a quandle, any constant map with a constant value $c \in X$ is a reflection for $\left(X, r_\triangleright\right)$.
\item If $(X, \triangleright)$ is the cyclic rack, since one can easily deduce that $\End(X,\triangleright)
    = \LMlt(X,\triangleright)$, then $\K(X) = \LMlt(X,\triangleright)$. 
\end{enumerate}  
\end{exs}

Another consequence of the previous corollaries is that bijective reflections of a left rack allow us to obtain new solutions of the Yang-Baxter equation, as we explain in the following. 
\begin{prop}\label{soluzione_assoc_rack}
    Let $(X, \triangleright)$ be a left rack and let us consider $\kappa\in \Kbij(X, r_\triangleright)$. Then, the map
    $r_{\kappa}:X\times X\to X\times X$ defined by
    \begin{align*}
      r_{\kappa}\left(x,\,y\right)
        = \left(\kappa\left(y\right),\, \kappa^{-1}\left(y \triangleright x\right)\right),
    \end{align*}
    for all $x,y\in X$, is a bijective and non-degenerate solution on $X$. In addition, it holds that $\omega \in \K(X, r_\kappa)$ if and only if $\kappa\omega=\omega\kappa$.
    \begin{proof}
        The proof of the fact that $r_\kappa$ is a bijective and non-degenerate solution is a consequence of \cite[Theorem 2.15]{DoRySt24x} and \cref{refl_rack_char}.   
    Moreover, the rest of the claim follows from \cref{cor:refl_bilnd}.
    \end{proof}
\end{prop}

 \noindent Notice that the solution $(X, r_\kappa)$ in \cref{soluzione_assoc_rack} is equivalent to $(X,r_\triangleright)$ if and only if $\kappa=\id_X$. An easy calculation shows that $\triangleright = \triangleright_{r_\kappa}$ and therefore $r_{\triangleright}$ and $r_\kappa$ share the same left rack structure.
According to \cite[Lemma 2.25]{DoRySt24x}, it follows that $r_{\triangleright}$ and $r_\kappa$ are Drinfel'd-isomorphic, with the explicit Drinfel'd-isomorphism $\Phi:X \times X \to X \times X$ given by
$\Phi(x,y) = (x,\kappa(y))$, for all $x, y \in X$.
\medskip

The $\lambda$-centralizing reflections for bijective left non-degenerate solutions can be directly studied using the \emph{centralizer} of $\LMlt\left(X,\triangleright\right)$, as we will show in the result below. 
Hereafter, with $C_M(S)$ we will denote the centralizer in a monoid $M$ of a subset $S$.

\begin{theor}\label{bij_refl:bij_nondeg_sol}
    Let $(X,r)$ be a bijective left non-degenerate solution and $\kappa:X \to X$ a $\lambda$-centralizing map.
    Then the following hold:
    \begin{enumerate}
        \item If $\kappa \in C_{\End(X,\triangleright)}\left(\LMlt\left(X,\triangleright\right)\right)$, then $\kappa \in \K(X,r)$;
        \item If $\kappa \in \Kinj(X,r)$, then $\kappa \in C_{\End(X,\triangleright)}\left(\LMlt\left(X,\triangleright\right)\right)$;
        \item In particular, $\kappa \in \Kbij(X,r)$ if and only if $\kappa \in C_{\Aut(X,\triangleright)}\left(\LMlt\left(X,\triangleright\right)\right)$.
    \end{enumerate}
    \begin{proof}
        The first statement directly follows by \cref{cor:refl_bilnd}. Now, assume that $\kappa$ is an injective reflection.
        Again, from property $1.$ we can deduce that $L_x = L_{\kappa(x)}$, from which it follows that $L_x\kappa = L_{\kappa(x)}\kappa = \kappa L_x$, for all $x \in X$. Clearly, the last point of the assertion is a consequence of the former two.
    \end{proof}
\end{theor}

\begin{cor}\label{centraliz}
   If $(X,\triangleright)$ is a left rack, then
   \begin{align*}  C_{\End(X,\triangleright)}\left(\LMlt\left(X,\triangleright\right)\right) \subseteq \K(X)\,         \quad\text{and}\quad
        \Kinj(X)\subseteq C_{\End(X,\triangleright)}\left(\LMlt\left(X,\triangleright\right)\right).
   \end{align*}
   In particular, $C_{\Aut(X,\triangleright)}\left(\LMlt\left(X,\triangleright\right)\right)  = \Kbij(X)$.
\end{cor}
\smallskip

As an easy further consequence, we state the following result.
\begin{cor}\label{Lx-commuta}
    Let $(X,\triangleright)$ be a left rack and $a \in X$. Then, the map $L_a \in \K(X)$
    if and only if $L_a\in \text{Z}\left(\LMlt\left(X,\triangleright\right)\right)$, the center of $\LMlt(X, \triangleright)$.
\end{cor}

\smallskip

\begin{ex}\label{ex:dihedral}
Let $X=\mathbb{Z}_n$ and $(X,\triangleright)$ be the \textit{dihedral quandle} of order $n \in \mathbb{N}$.
        In \cite[Theorem 2.1]{ElMaRe12} the authors prove that $\Aut(X) = \Aff(\Z_n)$, the group of affine transformations of $\Z_n$.
        With an analogous proof one can compute $\End(X)$ and show that
        \[
            \End(X) = \left\{ f_{b,a}:\Z_n \to \Z_n \mid f_{b,a}(x) = b+ax,\, x,a,b \in \Z_n\right\}.
        \]
        In light of \cref{refl_rack_char}, a map $\kappa:X \to X$ is a reflection for $(X,r_{\triangleright})$ if and only if there exist $a,b\in\Z_n$ such that $\kappa = f_{b,a}$ and it satisfies $\kappa L_x = \kappa L_{\kappa(x)}$, for all $x \in \Z_n$.
        Therefore, a straightforward computation shows that $\kappa \in \K(X,r_{\triangleright})$ if and only if 
        \begin{equation*}
            2a(a-1)x+2ab = 0 \text{ (mod $n$)}\,
        \end{equation*}
        holds, for all $x\in\Z_n$. Evidently, this is equivalent to both requiring that 
        \[
            \begin{cases}
                2ab = 0 &\text{ (mod $n$)}\,, \\
                2a(a-1) = 0 &\text{ (mod $n$)}\,.
            \end{cases}
        \]
        The distinct solutions of such a system will provide all reflections $\kappa$ for $(X,r_{\triangleright})$.
        For example if $\kappa$ is bijective, i.e. if $a \in \Z_n^{*}$ is invertible, then the following cases occur:
        \begin{itemize}
            \item[-] if $\gcd(n,4) = 1$ then $(b,a) = (0,1)$;
            \item[-] if $\gcd(n,4) = 2$ then $(b,a) = (0,1), \left(\frac{n}{2},1\right)$;
            \item[-] if $\gcd(n,4) = 4$ then $(b,a) = (0,1), \left(0, \frac{n}{2}+1\right), \left(\frac{n}{2},1\right), \left(\frac{n}{2}, \frac{n}{2}+1\right)$.
        \end{itemize}
        In particular, $|\Kbij(X,\triangleright)| = \gcd(n, 4)$ and all bijective reflections are involutive. \\
    An analogous procedure can be carried out to determine that the only reflections for the left derived solution associated to the infinite dihedral quandle are the identity map and the maps of arbitrary constant value.
\end{ex}
\smallskip

\begin{rem}\label{red-contrex}
Note that the converse of \cref{cor:refl_bilnd}-$1.$ is not true. Indeed, the map
\[
    f_{1,-1}:\Z_3 \to \Z_3\,, f_{1,-1}(x) = 1-x
\]
is an involutive automorphism of the dihedral quandle on $\Z_3$ (see \cref{ex:dihedral}), but it is not a reflection for the associated left derived solution since, in this case, the unique reflection is the identity map.
\end{rem}

\smallskip

Let us conclude this section with some considerations on idempotent left non-degenerate solutions. Observe that the shelf associated with any idempotent solution $(X, r)$ is the trivial shelf given by $L_x= x$, for every $x\in X$.


\begin{prop}
    Let $(X,r)$ be an idempotent left non-degenerate solution. Then:
    \begin{enumerate}
        \item
        a $\lambda$-centralizing   map $\kappa:X \to X$ is a reflection if and only if\, $\kappa^2 = \kappa$;
        
        \item $\kappa:X \to X$ is a $\rho$-invariant reflection if and only if $\kappa = \id_X$.
    \end{enumerate}
\begin{proof}
Let us initially observe that the first statement is a direct consequence of \cref{prop:refl-lnd}.
%
For the second one, note that \eqref{q} is satisfied if and only if\ $\kappa\rho_y = \rho_y$, for every $y\in X$. 
Consequently, \eqref{t} holds if and only  if $\lambda_x\kappa = \lambda_x$, for every $x\in X$, and, since $r$ is left non-degenerate, we get $\kappa = \id_X$.
\end{proof}
\end{prop}



\medskip

\section{Reflections for right-non degenerate solutions}
In this section, we investigate reflections for right non-degenerate solutions by shifting the attention to the class of $\rho$-invariant reflections, that is, $\rho_x=\rho_{\kappa(x)}$ holds, for all $x \in X$. 

\smallskip
\begin{theor}\label{prop:refl-rnd}
    Let $(X,r)$ be a right non-degenerate solution, $(X, \triangleleft_r)$ its associated shelf, and $\kappa:X \to X$ a $\rho$-invariant map.
    Then, $\kappa \in \K(X,r)$ if and only if $\kappa R_x=R_x \kappa$, for all $x \in X$.
    \begin{proof}
Let $x, y \in X$. Using \cref{sol:lnd}, we get $\kappa q(x,y)= \kappa \rho_{\rho_y(x)}\lambda_x(y)=\kappa R_{\rho_y(x)}(y)$ and, on the other hand,
\begin{align*}
    q(x, \kappa(y))=\rho_{\kappa\rho_{\kappa(y)}(x)}\lambda_x\kappa(y)=R_{\rho_{\kappa(y)}(x)}\kappa(y)=R_{\rho_y(x)}\kappa(y).
\end{align*}
Moreover, it holds that
\begin{align*}
    t(x,y)=\rho^{-1}_{\kappa\rho_{\kappa\rho_y(x)}\lambda_x(y)}R_{\kappa\rho_y(x)\lambda_x(y)}\kappa\rho_y(x)= \rho^{-1}_{\kappa q(x,y)}R_{\kappa q(x,y)}\kappa\rho_y(x)
\end{align*}
and, similarly, $t(x, \kappa(y))=\rho^{-1}_{q(x, \kappa(y))}R_{q(x, \kappa(y))}\kappa \rho_y(x)$. Thus, \eqref{t} depends on equation \eqref{q}.
Therefore, to get the claim it is enough to observe that \eqref{q} is equivalent to requiring that $\kappa$ commutes with every right multiplication $R_x$, for all $x \in X$.
    \end{proof}   
\end{theor}

Below, we give a list of consequences of \cref{prop:refl-rnd}. First, let us remark that the previous result extends \cref{right} to the class of arbitrary right non-degenerate solutions which are not necessarily involutive. 
This follows from the fact that if $(X,r)$ is a right non-degenerate involutive solution then $(X,\triangleleft_r)$ is trivial.

\begin{cor}
Let $\left(X, r\right)$ be a right non-degenerate solution. If $\kappa,\omega:X\to X$ are $\rho$-invariant reflections for $\left(X,r\right)$, then $\kappa\omega \in \K\left(X, r\right)$ and $\kappa\omega$ is $\rho$-invariant.
\end{cor}

\begin{cor}\label{cor:right-der}
Let $(X,r)$ be a right non-degenerate solution and let us assume that $\kappa$ is a $\rho$-invariant map. Then, $\kappa \in \K(X, r)$ if and only if $\kappa \in \K\left(X, r_\triangleleft  \right)$. 
    \begin{proof}
        The claim follows by \cref{YBE_and_shelf} and \cref{prop:refl-rnd}.
    \end{proof}
\end{cor}

As in the case of left non-degenerate solutions, we can formulate the statement in \cref{prop:refl-rnd}
by means of the right multiplication group of the right rack structure.
\begin{cor}\label{cor_right-bij}
    Let $(X,r)$ be a bijective right non-degenerate solution, $(X, \triangleleft_r)$ its associated right rack, and $\kappa:X \to X$ a $\rho$-invariant map.
    Then, the following hold:
    \begin{enumerate}
        \item $\kappa \in \K(X,r) \iff \kappa \in C_{\Map_X}\left(\RMlt\left(X,\triangleleft_r\right)\right)$.
        \item $\kappa \in \Kbij(X,r) \iff \kappa \in C_{\Sym _X}\left( \RMlt\left(X,\triangleleft_r\right)\right)$.
    \end{enumerate}
\end{cor}

\smallskip


\begin{ex}
Let $X = \Z_n$ and consider $(X,\triangleleft)$ the right dihedral quandle of order $n\in\mathbb{N}$.
    The \cref{cor_right-bij} implies that a map $\kappa \in \K\left(X,r_\triangleleft\right)$ if and only if the following holds for all $x,y \in X$:
    \begin{equation}\label{id_right_dihed}
        \kappa(x)+\kappa(2y-x) = 2y\,. 
    \end{equation}
    Such a map $\kappa$ has to satisfy $\kappa(2y) = 2y-\kappa(0)$ and $\kappa(2y-1)=2y-\kappa(1)$, for all $y\in X$, therefore it is parametrized by its values in $0$ and $1$.
    To determine their values, we notice that $2\kappa(0) = 0$ and $2(\kappa(1)-1)=0$, as a consequence of \eqref{id_right_dihed}.
    Therefore, all reflections for $(X,r_{\triangleleft})$ can be classified as follows:
    \begin{itemize}
        \item[-] if $n$ is odd then $\kappa = \id_X$;
        \item[-] if $n$ is even then $\kappa(0) \in \{0,\frac{n}{2}\}$, $\kappa(1) \in \{1, \frac{n}{2}+1\}$ and $\kappa$ is determined as above.
    \end{itemize}
    Analogously, one can prove that the right derived solution associated with a right infinite dihedral quandle admits the identity map as its unique reflection.
\end{ex}

\begin{prop}
    \label{moltipld_rifl_shelf} Let $(X,\triangleleft)$ be a right shelf, $(X,r_{\triangleleft})$ its associated solution, and $a\in X$. Then, the following conditions are equivalent:
    \begin{enumerate}
        \item $R_a$ is a reflection for $(X,r_\triangleleft)$;
        \item $R_aR_x = R_xR_a$, for all $x\in X$;
        \item $R_{R_a(x)}R_a=R_xR_a $, for all $x\in X$.
    \end{enumerate}
    \begin{proof}
        The equivalence of $1.$ and $2.$ directly follows from \cref{prop:refl-rnd}, while the equivalence of $2.$ and $3.$ easily stems from the definition of right shelf.
    \end{proof}
\end{prop}

 \begin{rem}
    As a direct consequence of \cref{moltipld_rifl_shelf}, if $(X, \triangleleft)$ is a right rack, then $Z(\RMlt(X, \triangleleft)) \subseteq \K(X)$. Moreover, if $R_a$ is a reflection then $R_a\in\End\left(X,\triangleleft\right)$.
   In particular, if the right multiplication group is abelian - as in the case of a cyclic rack - then every right multiplication provides a reflection for the associated solution.
 \end{rem}

\smallskip

Let us conclude with an observation on idempotent right non-degenerate solutions.
\begin{rem}
    Note that if $(X,r)$ is an idempotent right non-degenerate solution, one can check that the unique $\rho$-invariant reflection is the identity $\id_X$. On the other hand, the unique bijective $\lambda$-centralizing reflection for $(X, r)$ is the identity $\id_X$.
\end{rem}

\medskip

 \section{Reflections for bijective non-degenerate solution}
This section is devoted to a global analysis of reflections for bijective non-degenerate solutions in light of results obtained in the previous two sections.
\medskip

Let us start with some considerations on involutive non-degenerate solutions $(X, r)$. One can easily show that any $\rho$-invariant reflection can be described in terms of the retract solution of $(X, r)$. Recall that the equivalence relation of \emph{retraction} on $X$ is defined by setting $x \sim y \iff \lambda_x = \lambda_y$, for all $x,y \in X$. 
Thus, there is a natural induced solution $\Ret(X,r)$ on the quotient $\tilde{X}:= X/\sim$, called the \emph{retract} of $(X, r)$, which itself is involutive and non-degenerate. In the sequel, $[x]_{\sim}$ denotes the equivalence class of any $x \in X$.
For more information, we refer the reader to the standard literature (see \cite[Section 3]{ESS99}).

\begin{rem}\label{rem_rho_T}
The retraction relation for an involutive non-degenerate solution $(X, r)$ can be equivalently defined by requiring that $x \sim y \iff \rho_x = \rho_y$, for all $x,y \in X$. This follows directly by the identity $\rho_x=T\lambda^{-1}_xT^{-1}$, for all $x \in X$, where $T: X \to X$ is the bijective map defined by $T(x)=\rho_x^{-1}(x)$, for all $x \in X$ (cf. \cite[Proposition 2.2]{ESS99}).
\end{rem}

\begin{prop}\label{retrazione}
    Let $(X,r)$ be an involutive non-degenerate solution. Then, the following hold:
    \begin{enumerate}
        \item Any map $\kappa:X\to X$ such that $\kappa(x)\in [x]_{\sim}$, for every $x\in X$, is a $\rho$-invariant reflection for $(X,r)$. In particular, all such reflections can be constructed in this manner.
    \item If $\kappa \in \K(X, r)$ is $\rho$-invariant, then it induces a reflection $\tilde{\kappa}:\tilde{X}\to\tilde{X}$ of the retraction $\Ret(X, r)$ such that $\tilde{\kappa}([x]_{\sim}) = [\kappa(x)]_{\sim}$, for all $x \in X$.
    \end{enumerate}
    \begin{proof}
        The two statements are a direct consequence of \cref{right}. We added the proof in the file.
    \end{proof}
\end{prop}


\medskip

Now, we aim to combine the results obtained in the previous sections and analyze what happens in the special case of bijective non-degenerate solutions. To do so, in \cite[cf. Proposition 3.9]{LeVe19}, it is shown the existing link between left and right structure racks (see also \cite[Remark 3.10]{LeVe19}).

\begin{lemma}\label{bijnondeg_rack}
    Let $(X, r)$ be a bijective non-degenerate solution. Then, the left and the right structure racks of $(X, r)$ are anti-isomorphic via the bijection $T: X \to X$ defined by $T(x)=\rho^{-1}_x(x)$, for all $x \in X$. Concretely, the relation
    \begin{align*}
        T\left( x \triangleright_r y\right)= T(y) \triangleleft_r T(x)
    \end{align*}
    holds, for all $x, y \in X$.
    Equivalently, $L_x = T^{-1}R_{T(x)}T$, where $T^{-1}(x)=\lambda_x^{-1}\left(x \triangleright_r x \right)$, for all $x \in X$.
\end{lemma} 

\begin{rem}
   To be thorough, let us observe that if $\left(X, r\right)$ is a bijective non-degenerate square-free solution, i.e., $r\left(x,x\right) = \left(x,x\right)$, for all $x \in X$, then $\triangleright_r = \triangleleft_r^{op}$.
    The converse is not necessarily true (consider for example a non-trivial Lyubashenko's solution).\\
    However, in general, it is easy to check that two solutions of derived type $\left(X, r_{\triangleright}\right)$ and $\left(Y, r_{\triangleleft}\right)$ are isomorphic if and only if $L_a = \id_X$, for every $a\in X$, and $R_u = \id_Y$, for every $u\in Y$. In particular, this happens only for involutive solutions.
\end{rem}

As an easy consequence of \cref{bijnondeg_rack}, since any anti-homomorphism between $\left(X, \triangleright_r\right)$ and $\left(X, \triangleleft_r\right)$ satisfies $L_x = F^{-1}R_{F\left(x\right)}F$, for every $x\in X$, it is a routine computation to check the following.
\begin{theor}\label{prop:refl-bnsol}
    Let $(X,r)$ be a bijective non-degenerate solution, $\kappa:X \to X$ a map, $F:X\to X$ an anti-homomorphism from $\left(X, \triangleright_r\right)$ in $\left(X, \triangleleft_r\right)$, and
    $\omega, \bar{\omega}:X\to X$ the maps defined by $\omega:= F\kappa F^{-1}$ and $\bar{\omega}:= F^{-1}\kappa F$.
    Then, the following hold:
    \begin{enumerate}
        \item if $\kappa$ is $\lambda$-centralizing, then, 
        \begin{align*}
            \kappa\in \K(X,r) \Leftrightarrow \ 
            \begin{cases}
            \ \forall\, x\in X \quad \kappa L_x = \kappa L_{\kappa(x)}& \\
            \ \  \kappa\in\End\left(X, \triangleright_r\right)
        \end{cases}
        \Leftrightarrow  \
        \begin{cases}
            \ \forall\, x\in X \quad \omega R_x = \omega R_{\omega(x)}& \\
            \ \ \omega\in\End\left(X, \triangleleft_r\right)
        \end{cases}
        \end{align*}
        \item if $\kappa$ is $\rho$-invariant, then 
        \begin{align*}
        \kappa\in \K(X,r) \ \Leftrightarrow \
        \forall\, x\in X \quad \kappa R_x = R_x\kappa
        \ \Leftrightarrow \
        \forall\, x\in X \quad \bar{\omega}L_x = L_x\bar{\omega}.
    \end{align*}
    \end{enumerate}
\end{theor}
\smallskip

If we specialize the \cref{prop:refl-bnsol} to a solution $(X, r)$ of derived type, it does not yield a novel result, as we explain below. One can argue analogously if $(X,r)$ is a solution of right derived type.
\begin{rem} 
If $(X, r)$ is a solution whose left derived solution is associated with a left rack $(X,\triangleright)$, then its right derived solution $(X,r_{\triangleleft})$ has the form $r_{\triangleleft}(x,y) = (y \triangleleft_r x,x) = (x \triangleright y, x)$, for all $x,y \in X$, whence $\triangleleft_r = \triangleright^{\text{op}}$.
Therefore, the map $T$ from \cref{bijnondeg_rack} actually is an automorphism of $(X,\triangleright)$.
If we notice that $\Aut(X,r) = \Aut(X,\triangleright)$ then \cref{prop:refl-bnsol} simply recovers the observation made in \cref{iso:refl}.
\end{rem}

\begin{cor}\label{rem:maps-lambda-rho-cent-inv}
Let $(X, r)$ be a bijective non-degenerate solution and let $\kappa: X \to X$ be a $\rho$-centralizing and $\rho$-invariant map. Then, the following holds
\begin{align*}
    \kappa\in\K(X, r) \iff \kappa\text{ is $\lambda$-centralizing}\,.
\end{align*}
\begin{proof}
    First, assume that $\kappa\in \K(X,r)$.
    Then, since by assumption $\kappa$ is $\rho$-invariant, \cref{prop:refl-bnsol}-2 ensures that $\kappa$ commutes with all $R_x$.
    As a consequence, for all $x,y \in X$ the following holds:
    \begin{align*}
        \kappa\lambda_x(y) &= \kappa\rho^{-1}_{\rho_y(x)}R_{\rho_y(x)}(y) = \rho^{-1}_{\rho_y(x)}R_{\rho_y(x)}\kappa(y) = \lambda_x\kappa(y)\,,
    \end{align*}
    proving the first part.
    Conversely, the assertion easily follows from \cref{prop:refl-bnsol}-2. since the hypothesis yields that $\kappa$ satisfies $\kappa R_x =R_x\kappa$, for every $x\in X$. 
\end{proof}
\end{cor}

\smallskip

In a general setting and inspired by \cite[Proposition 3.5]{LeVe22}, we show how to obtain reflections for arbitrary solutions starting from a reflection and maps that are $\lambda, \rho$-centralizing and $\lambda, \rho$-invariant. 

\begin{theor}\label{prop:gen_caseLV22}
    Let $\left(X,r\right)$ be a solution, $\kappa\in \K\left(X,r\right)$, and $\varphi, \psi:X\to X$ maps that are $\lambda, \rho$-centralizing and $\lambda, \rho$-invariant. Then, the map $\omega:= \varphi\kappa\psi\in\K\left(X,r\right)$.
    \begin{proof}
        Initially,  observe that
    \begin{itemize}
        \item[-] $\varphi$ is $\lambda$-centralizing and $\rho$-invariant if and only if \ $r\left(\id_X\times\varphi\right)= \left(\varphi\times\id_X\right)r$;
        \item[-] $\varphi$ is $\rho$-centralizing and $\lambda$-invariant if and only if \ 
        $\left(\id_X\times\varphi\right)r = r\left(\varphi\times\id_X\right)$.
    \end{itemize}
  Similarly, the equivalences hold for the map $\psi$.  By these equivalences, we get
\begin{align*}
    r\left(\id_X\times \omega\right)r\left(\id_X\times \omega\right) &=\left(\varphi\times\varphi\right)
    r\left(\id_X\times \kappa\right)r \left(\id_X\times \kappa\right)
     \left(\psi\times\psi\right)\\
\left(\id_X\times \omega\right)r\left(\id_X\times \omega\right)r &= 
    \left(\varphi\times\varphi\right)
    \left(\id_X\times \kappa\right)r \left(\id_X\times \kappa\right)r
     \left(\psi\times\psi\right).
\end{align*}
Therefore, since $\kappa$ is a reflection, then using \eqref{RE}, $\omega$ is as well.
    \end{proof}
\end{theor}

\begin{rem}
    Notice that if $(X, r)$ is an involutive solution, then, for the maps $\varphi, \psi$, the assumptions of being $\lambda$-invariant and $\rho$-centralizing are redundant. Hence, the result in  \cite[Proposition 3.5]{LeVe22} is a special case of \cref{prop:gen_caseLV22}. 
\end{rem}

\smallskip

The following is an example of a bijective non-degenerate solution admitting reflections which are $\lambda$-centralizing, $\rho$-invariant and neither of both.

\begin{ex}
    Let $X=\{1,\dots,8\}$ and consider the bijective non-degenerate solution $(X, r)$ given by
    \begin{align*}
        \lambda_1 = \lambda_2 = \lambda_7 = \lambda_8 = \id_X \,, \quad
        &\lambda_3 = \lambda_4 = \lambda_5 = \lambda_6 = (35)(46)(78) \,, \\
        \rho_1 = \rho_2 = \id_X \,,\quad \rho_3 = \rho_4 = (3645) \,,\quad &\rho_5 = \rho_6 = (3546) \,,\quad \rho_7 = \rho_8 = (34)(56) \,.
    \end{align*}
    With the aid of the \texttt{GAP} software \cite{GAP4}, we determined that there are 128 reflections for $(X,r)$.
    For example, among these we have the following:
    \begin{itemize}
        \item[-] $\kappa_1 = 21222211$ is a $\lambda$-centralizing reflection which is not $\rho$-invariant;
        \item[-] $\kappa_2 = 11346578$ is $\rho$-invariant reflection which is not $\lambda$-centralizing;
        \item[-] $\kappa_3 = 21436578$ is a $\lambda$-centralizing and $\rho$-invariant reflection;
        \item[-] $\kappa_4 = 12556611$ is a reflection which is neither $\lambda$-centralizing nor $\rho$-invariant.
    \end{itemize}
    For the first three examples we employed \cref{prop:refl-bnsol} to verify that they actually are reflections. Moreover, there are 64 which are neither $\lambda$-centralizing nor $\rho$-invariant.
    However, by applying \cref{prop:gen_caseLV22}, we can recover each of them in the form $\varphi\kappa\psi$ for some $\lambda,\rho$-centralizing and $\lambda,\rho$-invariant maps $\varphi$, $\psi$ (different from the identity) and for some $\kappa \in \K(X,r)$.
    For instance, if $\kappa_4$ is as in the example above, then $\kappa_4 = \varphi\kappa\psi$ where $\varphi = 21436587 = \psi$ and $\kappa = 12556611$.
\end{ex}

\smallskip

Finally, we provide a class of bijective non-degenerate solutions, including the solution in \cref{soluzione_assoc_rack}, whose reflections are not necessarily $\lambda$-centralizing or $\rho$-invariant.

\begin{ex}
    Let $(X,\triangleright)$ a left rack and $\varphi \in \Aut(X,\triangleright)$.
    Then, \cite[Theorem 2.15]{DoRySt24x} ensures that the map 
  $r(x,y) 
        =\left(\varphi(y),L_y\varphi^{-1}(x)\right)\,$, 
    for all $x,y \in X$, is a bijective non-degenerate solution.
    An easy application of \cref{lemma_rifl_qt} allows to determine that 
    \begin{center}
  $ \kappa\in \K(X,r) \iff \ 
            \begin{cases}
            \ \forall\, x\in X \quad \kappa L_x = \kappa L_{\kappa(x)}& \\
            \ \  \kappa\in\End\left(X, \triangleright\right)     \end{cases}$
  \end{center}  
  In particular, observe that 
  $\kappa$ is $\lambda$-centralizing if and only if $\kappa \varphi = \varphi \kappa$.   On the other hand, $\kappa$ is $\rho$-invariant if and only if $\kappa$ is $L$-invariant.
  For instance, let $(X,\triangleright)$ be the dihedral quandle of order $n \neq 1,2$ on $\Z_n$, $\kappa$ the constant map of constant value $0 \in \Z_n$, and notice that $\kappa \in \K(X)$ is neither $\lambda$-centralizing nor $\rho$-invariant.
\end{ex}

For other observations, we refer the reader to the appendix \cref{sec:appendix} of this document.

\medskip

\section{Constructions of reflections}
One can find in the literature some techniques to construct solutions from already known solutions, such as the \emph{twisted unions} of involutive non-degenerate solutions \cite{ESS99}, the \emph{matched product} \cite{CCoSt20}, and the \emph{strong semilattice} of solutions \cite{CCoSt20x-2}. In this context, this section is dedicated to providing methods for constructing reflections of these solutions starting from reflections on the known solutions.

\subsection{The reflection of the matched product of two solutions}

Let us begin by recalling the essential notions contained in \cite{CCoSt20}. 
If $(S,r_S)$ and $(T, r_T)$ are solutions that, hereinafter,  we write as $r_S\left(a, b\right) = \left(\lambda_a\left(b\right), \rho_b\left(a\right)\right)$, for all $a,b\in S$, and $r_T\left(u,v\right) = \left(\lambda_u\left(v\right),\rho_v\left(u\right)\right)$, for all $u,v\in T$, and $\alpha:T\to\Sym\left(S\right)$ and $\beta:S\to\Sym\left(T\right)$ are maps, set $\alpha_u:=\alpha\left(u\right)$, for every $u \in T$, and $\beta_a:=\beta\left(a\right)$, for every $a\in S$. Then, the quadruple $\left(r_S,r_T, \alpha,\beta\right)$ is said to be a \emph{matched product system of solutions} if the following conditions hold
	{\small
		\begin{center}
			\begin{minipage}[b]{.5\textwidth}
				\vspace{-\baselineskip}
				\begin{align}\label{eq:primo}\tag{s1}
					\alpha_u\alpha_v = \alpha_{\lambda_u\left(v\right)}\alpha_{\rho_{v}\left(u\right)}
				\end{align}
			\end{minipage}%
			\hfill\hfill\hfill
			\begin{minipage}[b]{.5\textwidth}
				\vspace{-\baselineskip}
				\begin{align}\label{eq:secondo}\tag{s2}
					\beta_a\beta_b=\beta_{\lambda_a\left(b\right)}\beta_{\rho_b\left(a\right)}
				\end{align}
			\end{minipage}
		\end{center}
		\begin{center}
			\begin{minipage}[b]{.5\textwidth}
				\vspace{-\baselineskip}
				\begin{align}\label{eq:quinto}\tag{s3}
					\rho_{\alpha^{-1}_u\!\left(b\right)}\alpha^{-1}_{\beta_a\left(u\right)}\left(a\right) = \alpha^{-1}_{\beta_{\rho_b\left(a\right)}\beta^{-1}_b\left(u\right)}\rho_b\left(a\right)
				\end{align}
			\end{minipage}%
			\hfill\hfill
			\begin{minipage}[b]{.5\textwidth}
				\vspace{-\baselineskip}
				\begin{align}\label{eq:sesto}\tag{s4}
					\rho_{\beta^{-1}_a\!\left(v\right)}\beta^{-1}_{\alpha_u\left(a\right)}\left(u\right) = \beta^{-1}_{\alpha_{\rho_v\left(u\right)}\alpha^{-1}_v\left(a\right)}\rho_v\left(u\right)
				\end{align}
			\end{minipage}
		\end{center}
		\begin{center}
			\begin{minipage}[b]{.5\textwidth}
				\vspace{-\baselineskip}
				\begin{align}\label{eq:terzo}\tag{s5}
					\lambda_a\alphaa{}{\beta^{-1}_{a}\left(u\right)}{}= \alphaa{}{u}{\lambdaa{\alphaa{-1}{u}{\left(a\right)}}{}}
				\end{align}
			\end{minipage}%
			\hfill\hfill
			\begin{minipage}[b]{.5\textwidth}
				\vspace{-\baselineskip}
				\begin{align}\label{eq:quarto}\tag{s6}
					\lambdaa{u}{\betaa{}{\alphaa{-1}{u}{\left(a\right)}}{}}=\betaa{}{a}{\lambdaa{\beta^{-1}_{a}\left(u\right)}{}}
				\end{align}
			\end{minipage}
		\end{center}
	}
	\noindent for all $a,b \in S$ and $u,v \in T$.\\
	As shown in \cite[Theorem 2]{CCoSt20}, any matched product system of solutions $\left(r_S, \, r_T,\, \alpha,\,\beta\right)$ determines a new solution  $r_S\bowtie r_T:(S{ \times} T)^2\to (S{ \times} T)^2$ defined by
		\begin{align*}
			&r_S\bowtie r_T\left(\left(a, u\right), 
			\left(b, v\right)\right) := 
			\left(\left(\alphaa{}{u}{\lambdaa{\bar{a}}{\left(b\right)}},\, \beta_a\lambdaa{\bar{u}}{\left(v\right)}\right),\ \left(\alphaa{-1}{\overline{U}}{\rhoo{\alphaa{}{\bar{u}}{\left(b\right)}}{\left(a\right)}},\,  \beta^{-1}_{\overline{A}}\rhoo{\beta_{\bar{a}}\left(v\right)}{\left(u\right)}\right) \right),
		\end{align*}
		\noindent where we set
		\begin{center}
		   $\bar{a}:=\alphaa{-1}{u}{\left(a\right)}$, \,\,$\bar{u}:= \beta^{-1}_{a}\left(u\right)$,\,\, $A:=\alphaa{}{u}{\lambdaa{\bar{a}}{\left(b\right)}}$,\, $U:=\beta_a\lambdaa{\bar{u}}{\left(v\right)}$,\,\, $\overline{A}:=\alphaa{-1}{U}{\left(A\right)}$,\,\, $\overline{U}:= \beta^{-1}_{A}\left(U\right)$,
		\end{center}
 		for all $\left(a,u\right),\left(b,v\right)\in S\times T$, called the \emph{matched product of the solutions} $r_S$ and $r_T$ (via $\alpha$ and $\beta$).\\
In particular, if the solutions $(S, r_S)$ and $(T, r_T)$ we start from are bijective with finite order, then $r_S\bowtie r_T$ is bijective with finite order. In addition, if $(S, r_S)$ and $(T, r_T)$ are non-degenerate, then $r_S\bowtie r_T$ is left non-degenerate. Thus, in the finite case, $r_S\bowtie r_T$ is also right non-degenerate.
		
\smallskip

In the following lemmas, consistently with maps $t$ and $q$ introduced in \cref{lemma_rifl_qt}, let us consider the maps $t,q$ from $\left(S\times T\right)\times\left(S\times T\right)$ defined by
$$
t\left(\left(a,u\right), \left(b,v\right)\right):= \lambda_{\lambda_{\left(a,u\right)}\left(b,v\right)}\left(\kappa\times \omega\right)\rho_{\left(b,v\right)}\left(a,u\right)
$$ 
and 
$$
q\left(\left(a,u\right), \left(b,v\right)\right):= \rho_{\left(\kappa\times \omega\right)\rho_{\left(b,v\right)}\left(a,u\right)}\lambda_{\left(a,u\right)}\left(b,v\right),
$$
for all $\left(a,u\right), \left(b,v\right)\in S\times T$, where $\kappa:S\to S$ and $\omega:T\to T$ are maps.

\begin{lemma}\label{lemma1}
	Let $\left(r_{S}, r_{T}, \alpha, \beta\right)$ be a matched product system of solutions, $\kappa:S\to S$ and $\omega:T\to T$ maps such that 
	\begin{align}\label{cond_match_ref}\tag{M}
	    \alpha_{u}\kappa &= \kappa\alpha_{u} \qquad \text{and} \qquad \beta_{a}\omega = \omega\beta_{a},
	\end{align} for all $u\in T$ and $a\in S$. 
 Set $\mathcal{A}:= \lambdaa{a}{\alphaa{}{\bar{u}}{\kappa\left(b\right)}}$
		and $\mathcal{U}:= \lambdaa{u}{\betaa{}{\bar{a}}{\omega\left(v\right)}}$, for all $a, b \in S$ and $u, v\in T$, then the following statements hold: 
	\begin{enumerate}
	    \item $t\left(\left(a,u\right), \left(\kappa\left(b\right),\omega\left(v\right)\right)\right)=
		\left(\lambdaa{\mathcal{A}}{\kappa\rhoo{\alphaa{}{\bar{u}}{\kappa\left(b\right)}}{\left(a\right)}}, 
		\lambdaa{\mathcal{U}}{\omega\rhoo{\betaa{}{\bar{a}}{\omega\left(v\right)}}{\left(u\right)}}\right)$,
		\item $t\left(\left(a,u\right), \left(b,v\right)\right)= \left(\lambdaa{A}{\kappa\rhoo{\alphaa{}{\bar{u}}{\left(b\right)}}{\left(a\right)}}, \lambdaa{U}{\omega\rhoo{\betaa{}{\bar{a}}{\left(v\right)}}{\left(u\right)}}\right),
	$ 
	\end{enumerate}
		for all $\left(a,u\right), \left(b,v\right)\in S\times T$.
	\begin{proof}
		Let $(a,u), (b,v) \in S \times T$. We show the proof for the first components of the identities in $1.$ and $2.$, since the proof for the second ones will be similar, by exchanging the roles of $a$ and $u$, $b$ and $v$, $\mathcal{A}$ and $\mathcal{U}$, $\alpha$ and $\beta$, and $\kappa$ and $\omega$, and by using $\left(s6\right)$ instead of $\left(s5\right)$. We have that \begin{align*}
		    t\left(\left(a,u\right), \left(\kappa\left(b\right),\omega\left(v\right)\right)\right)&=\lambda_{\left(\alpha_u\lambda_{\overline{a}}\kappa(b), \beta_a\lambda_{\overline u}\omega(v)\right)}\kappa \times \omega \left(\alpha_{\overline U}\rho_{\alpha_{\overline u}\kappa (b)}(a), \beta^{-1}_{\overline A}\rho_{\beta_a\omega(v)}(u)\right)\\
            &=\lambda_{\left(\lambda_a\alpha_{\overline{u}}\kappa(b), \lambda_u\beta_{\overline a}\omega(v)\right)}\kappa \times \omega \left(\alpha_{\overline U}\rho_{\alpha_{\overline u}\kappa (b)}(a), \beta^{-1}_{\overline A}\rho_{\beta_a\omega(v)}(u)\right)\\
&=\lambda_{\left(\mathcal{A}, \, \mathcal{U}\right)} \left(\kappa \alpha_{\overline U}\rho_{\alpha_{\overline u}\kappa (b)}(a), \omega\beta^{-1}_{\overline A}\rho_{\beta_a\omega(v)}(u)\right)
		\end{align*}
where in the second to last equality we use $\left(s5\right)$ and $\left(s6\right)$.
        Thus, the first component of $t\left(\left(a,u\right), \left(\kappa\left(b\right),\omega\left(v\right)\right)\right)$ is given by
  \begin{align*}
   \alphaa{}{\mathcal{U}}{\lambdaa{\overline{\mathcal{A}}}{\kappa\alphaa{-1}{\overline{\mathcal{U}}}{\rhoo{\alphaa{}{\bar{u}}{\kappa\left(b\right)}}{\left(a\right)}}}}\underset{\eqref{cond_match_ref}}{=}\alphaa{}{\mathcal{U}}{\lambdaa{\overline{\mathcal{A}}}{\alphaa{-1}{\overline{\mathcal{U}}}{\kappa\rhoo{\alphaa{}{\bar{u}}{\kappa\left(b\right)}}{\left(a\right)}}}}\underset{\left(s5\right)}{=}\lambdaa{\mathcal{A}}{\kappa\rhoo{\alphaa{}{\bar{u}}{\kappa\left(b\right)}}{\left(a\right)}}
  \end{align*}
and the first component of $t\left(\left(a,u\right), \left(b,v\right)\right)$ is
\begin{align*}
    \alphaa{}{U}{\lambdaa{\overline{A}}{\kappa\alphaa{-1}{\overline{U}}{\rhoo{\alphaa{}{\bar{u}}{\left(b\right)}}{\left(a\right)}}}}\underset{\eqref{cond_match_ref}}{=}  \alphaa{}{U}{\lambdaa{\overline{A}}{\alphaa{-1}{\overline{U}}{\kappa\rhoo{\alphaa{}{\bar{u}}{\left(b\right)}}{\left(a\right)}}}}\underset{\left(s5\right)}{=}\lambdaa{A}{\kappa\rhoo{\alphaa{}{\bar{u}}{\left(b\right)}}{\left(a\right)}},
\end{align*}
		which complete the proof.
	\end{proof}
\end{lemma}

\smallskip

\begin{lemma}\label{lemma2}
   Let $\left(r_{S}, r_{T}, \alpha, \beta\right)$ be a matched product system of solutions, $\kappa:S\to S$ and $\omega:T\to T$ maps such that \eqref{cond_match_ref} holds.  Set 
   \begin{align*}
       &\mathcal{A'}:= \alphaa{}{\mathcal{U}}{\lambdaa{\overline{\mathcal{A}}}{\left(\kappa\alphaa{-1}{\overline{\mathcal{U}}}{\rhoo{\alphaa{}{\bar{u}}{\kappa\left(b\right)}}{\left(a\right)}}\right)}} \qquad \mathcal{U'}:= \betaa{}{\mathcal{A}}{\lambdaa{\overline{\mathcal{U}}}{\left(\omega\betaa{-1}{\overline{\mathcal{A}}}{\rhoo{\betaa{}{\bar{a}}{\omega\left(v\right)}}{\left(u\right)}}\right)}}\\
       &\mathcal{A''}:= \alphaa{}{U}{\lambdaa{\overline{A}}{\left(\kappa\alphaa{-1}{\overline{U}}{\rhoo{\alphaa{}{\bar{u}}{\left(b\right)}}{\left(a\right)}}\right)}} \qquad \mathcal{U''}:= \betaa{}{A}{\lambdaa{\overline{U}}{\left(\omega\betaa{-1}{\overline{A}}{\rhoo{\betaa{}{\bar{a}}{\left(v\right)}}{\left(u\right)}}\right)}},
   \end{align*}for all $a\in S$ and $u\in T$.
		Then, the following statements hold:  
			\begin{enumerate}
	    \item $q\left( \left(a,u\right),\left(\kappa\left(b\right), \omega\left(v\right)\right)  \right)=\left(\alphaa{-1}{\overline{\mathcal{U'}}}{\rhoo{\kappa\rhoo{\kappa\alphaa{}{\bar{u}}{\left(b\right)}}{\left(a\right)}}{\lambdaa{a}{\kappa\alphaa{}{\bar{u}}{\left(b\right)}}}},
		\betaa{-1}{\overline{\mathcal{A'}}}{\rhoo{\omega\rhoo{\omega\betaa{}{\bar{a}}{\left(v\right)}}{\left(u\right)}}{\lambdaa{u}{\omega\betaa{}{\bar{a}}{\left(v\right)}}}}\right),
		$ 
		\item $\left(\kappa \times \omega \right)q\left(\left(a,u\right), \left(b,v\right)\right)= \left(\alphaa{-1}{\overline{\mathcal{U''}}}{\kappa\rhoo{\kappa\alphaa{}{\bar{u}}{\left(b\right)}}{\lambdaa{a}{\alphaa{}{\bar{u}}{\left(b\right)}}}},
		\betaa{-1}{\overline{\mathcal{A''}}}{\omega\rhoo{\omega\rhoo{\betaa{}{\bar{a}}{\left(v\right)}}{\left(u\right)}}{\lambdaa{u}{\betaa{}{\bar{a}}{\left(v\right)}}}}
		\right)$, 	
	\end{enumerate}
for all $\left(a,u\right), \left(b,v\right)\in S\times T$.
\begin{proof}
Similar to the proof of \cref{lemma1}, applying the conditions \eqref{cond_match_ref} and $\left(s5\right)$-$\left(s6\right)$.
\end{proof}
\end{lemma}

\medskip

\begin{theor}\label{teo_matched}
	Let $\left(r_{S}, r_{T}, \alpha, \beta\right)$ be a matched product system of solutions, $\kappa:S\to S$ and $\omega:T\to T$ maps such that 
	\begin{align}
	    \alpha_{u}\kappa = \kappa\alpha_{u} \qquad \text{and} \qquad \beta_{a}\omega = \omega\beta_{a} \tag{M}
	\end{align} for all $a\in S$ and $u\in T$. Then, the following statements are equivalent: 
	\begin{enumerate}
		\item[(i)] $\kappa \in \K\left(S, r_{S}\right)$ and $\omega \in \K\left(T, r_{T}\right)$;
		\item[(ii)] $\kappa\times \omega \in \K\left( S \bowtie T, r_{S}\bowtie
		r_{T}\right)$.
	\end{enumerate}
	\begin{proof}
Initially, suppose that $\kappa \in \K(S, r_{S})$ and $\omega \in \K(T,r_{T})$. If $(a, u), (b, v) \in S \times T$, we obtain
		\begin{align*}
		&t\left(\left(a,u\right), \left(\kappa\left(b\right),\omega\left(v\right)\right)\right)
		= \left(\lambdaa{\mathcal{A}}{\kappa\rhoo{\alphaa{}{\bar{u}}{\kappa\left(b\right)}}{\left(a\right)}}, 
		\lambdaa{\mathcal{U}}{h\rhoo{\betaa{}{\bar{a}}{h\left(v\right)}}{\left(u\right)}}\right) &\mbox{by \cref{lemma1}-$1.$}\\
		&= \left(\lambdaa{\lambdaa{a}{\kappa\alphaa{}{\bar{u}}{\left(b\right)}}}{\kappa\rhoo{\kappa\alphaa{}{\bar{u}}{\left(b\right)}}{\left(a\right)}}, 
		\lambdaa{\lambdaa{u}{\omega\betaa{}{\bar{a}}{\left(v\right)}}}{\omega\rhoo{\omega\betaa{}{\bar{a}}{\left(v\right)}}{\left(u\right)}}\right)&\mbox{by \eqref{cond_match_ref}}\\
		&=\left(t\left(a, \kappa\alpha_{\bar{u}}\left(b\right)\right), t\left(u, \omega\beta_{\bar{a}}\left(v\right)\right)\right)\\
		&= \left(t\left(a, \alpha_{\bar{u}}\left(b\right)\right), t\left(u, \beta_{\bar{a}}\left(v\right)\right)\right)&\mbox{by \eqref{t}-\eqref{q}}\\
		&=\left(\lambdaa{A}{\kappa\rhoo{\alphaa{}{\bar{u}}{\left(b\right)}}{\left(a\right)}}, \lambdaa{U}{\omega\rhoo{\betaa{}{\bar{a}}{\left(v\right)}}{\left(u\right)}}\right)\\
		&= t\left(\left(a,u\right), \left(b,v\right)\right) &\mbox{by \cref{lemma1}-$2.$}
		\end{align*}
Moreover, it follows from the last computations and \eqref{cond_match_ref} that 
		\begin{align*}
		\mathcal{A'} &= \alphaa{}{\mathcal{U}}{\lambdaa{\overline{A}}{\left(\kappa\alphaa{-1}{\overline{\mathcal{U}}}{\rhoo{\alphaa{}{\bar{u}}{\kappa\left(b\right)}}{\left(a\right)}}\right)}}
		=\lambdaa{A}{\kappa\rhoo{\alphaa{}{\bar{u}}{\left(b\right)}}{\left(a\right)}}=\alphaa{}{U}{\lambdaa{\overline{A}}{\left(\kappa\alphaa{-1}{\overline{U}}{\rhoo{\alphaa{}{\bar{u}}{\left(b\right)}}{\left(a\right)}}\right)}}=\mathcal{A''}
		\end{align*}
		and similarly $\mathcal{U'} = \mathcal{U''}$.
		Hence, we have
		\begin{align*}
		&q\left( \left(a,u\right),\left(\kappa\left(b\right), \omega\left(v\right)\right)  \right)\\
		&=\left(\alphaa{-1}{\overline{\mathcal{U'}}}{\rhoo{\kappa\rhoo{\kappa\alphaa{}{\bar{u}}{\left(b\right)}}{\left(a\right)}}{\lambdaa{a}{\kappa\alphaa{}{\bar{u}}{\left(b\right)}}}},
		\betaa{-1}{\overline{\mathcal{A'}}}{\rhoo{\omega\rhoo{\omega\betaa{}{\bar{a}}{\left(v\right)}}{\left(u\right)}}{\lambdaa{u}{\omega\betaa{}{\bar{a}}{\left(v\right)}}}}\right) &\mbox{by \cref{lemma2}-$1.$}\\
		&=\left(\alpha^{-1}_{\overline{\mathcal{U'}}}q\left(a, \kappa\alphaa{}{\bar{u}}{\left(b\right)}\right), \betaa{-1}{\overline{\mathcal{A'}}}q\left(u, \omega\betaa{}{\bar{a}}{\left(v\right)}\right)\right)\\
		&=\left(\alpha^{-1}_{\overline{\mathcal{U'}}}\kappa\left(q\left(a, \alphaa{}{\bar{u}}{\left(b\right)}\right)\right), \betaa{-1}{\overline{\mathcal{A'}}}\omega\left(q\left(u,\betaa{}{\bar{a}}{\left(v\right)} \right)\right)\right) &\mbox{by \eqref{t}-\eqref{q}}\\
		&=\left(\alpha^{-1}_{\overline{\mathcal{U''}}}\kappa\left(q\left(a, \alphaa{}{\bar{u}}{\left(b\right)}\right)\right), \betaa{-1}{\overline{\mathcal{A''}}}\omega\left(q\left(u,\betaa{}{\bar{a}}{\left(v\right)} \right)\right)\right) &\mbox{since $\mathcal{A'} = \mathcal{A''}$ and $\mathcal{U'} = \mathcal{U''}$}\\
		&= \left(\alphaa{-1}{\overline{\mathcal{U''}}}{\kappa\rhoo{\kappa\alphaa{}{\bar{u}}{\left(b\right)}}{\lambdaa{a}{\alphaa{}{\bar{u}}{\left(b\right)}}}},
		\betaa{-1}{\overline{\mathcal{A''}}}{h\rhoo{h\rhoo{\betaa{}{\bar{a}}{\left(v\right)}}{\left(u\right)}}{\lambdaa{u}{\betaa{}{\bar{a}}{\left(v\right)}}}}
		\right)\\
		&= \left(\kappa \times \omega \right)q\left(\left(a,u\right), \left(b,v\right)\right) &\mbox{by \cref{lemma2}-$2.$}
		\end{align*}
Thus, $\kappa\times \omega \in \K\left(S \bowtie T, r_{S}\bowtie r_{T}\right)$.\\
Conversely, we suppose that $\kappa\times \omega$ is a reflection of $r_{S}\bowtie r_{T}$ and we show that $\kappa \in \K(S,r_S)$. The proof for the map $\omega$ will be omitted since it is analogous.\\
Let $(a, u), (b, v) \in S \times T$ and
\begin{align*}
   & \mathcal{A} := \lambdaa{a}{\alphaa{}{\bar{u}}{\kappa\left(\alphaa{-1}{\bar{u}}{\left(b\right)}\right)}}\qquad A := \alphaa{}{u}{\lambdaa{\bar{a}}{\left(\alphaa{-1}{\bar{u}}{\left(b\right)}\right)}},
\end{align*}
thus from the previous computations, since $\kappa\times \omega$ is a reflection, we get
\begin{align*}
    \lambdaa{\mathcal{A}}{\kappa\rhoo{\alpha_{\bar{u}}\kappa\left(\alphaa{-1}{\bar{u}}{\left(b\right)}\right)}{\left(a\right)}}=\lambdaa{A}{\kappa\rhoo{\alpha_{\bar{u}}\alphaa{-1}{\bar{u}}{\left(b\right)}}{\left(a\right)}},
\end{align*}
namely $
    \lambdaa{\mathcal{A}}{\kappa\rho_{\kappa\left(b\right)}\left(a\right)}=\lambdaa{A}{\kappa\rhoo{{\left(b\right)}}{\left(a\right)}}.$
Thus,  since by \eqref{cond_match_ref} and $\left(s_5\right)$,  
\begin{center}
$\mathcal{A} =  \lambdaa{a}{\alphaa{}{\bar{u}}{\kappa\left(\alphaa{-1}{\bar{u}}{\left(b\right)}\right)}} = \lambdaa{a}{\kappa\left(b\right)}$
	\qquad 	$A = \alphaa{}{u}{\lambdaa{\bar{a}}{\left(\alphaa{-1}{\bar{u}}{\left(b\right)}\right)}} = \lambdaa{a}{\left(b\right)}$, \end{center}
we have that
		\begin{align*}
		t\left(a,\kappa\left(b\right) \right)=\lambdaa{\lambdaa{a}{\kappa\left(b\right)}}{\kappa\rhoo{\kappa\left(b\right)}{\left(a\right)}}
		=
		\lambdaa{\mathcal{A}}{\kappa\rhoo{\kappa\left(b\right)}{\left(a\right)}}
		= \lambdaa{A}{\kappa\rhoo{b}{\left(a\right)}}
		= \lambdaa{\lambdaa{a}{\left(b\right)}}{\kappa\rhoo{b}{\left(a\right)}}=t\left(a,b\right).
		\end{align*}
 Moreover,
since
\begin{align*}
    \rhoo{\kappa\rhoo{\kappa\alphaa{}{\bar{u}}{\left(\alpha^{-1}_{\bar{u}}\left(b\right)\right)}}{\left(a\right)}}{\lambdaa{a}{\kappa\alphaa{}{\bar{u}}{\left(\alpha^{-1}_{\bar{u}}\left(b\right)\right)}}}=\kappa\rhoo{\kappa\alphaa{}{\bar{u}}{\left(\alpha^{-1}_{\bar{u}}\left(b\right)\right)}}{\lambdaa{a}{\alphaa{}{\bar{u}}{\left(\alpha^{-1}_{\bar{u}}\left(b\right)\right)}}},
 \end{align*}
we obtain $q(a, \kappa(b))=\kappa q(a,b)$, i.e., $\kappa$ is a reflection for the solution $r_S$.
\end{proof}
\end{theor}
 
 \smallskip

According to \cite[Corollary 6]{CCoSt20}, in the case of left non-degenerate and involutive solutions, the requirements to have a matched product
system of solutions become easier. In particular,  a quadruple $\left(r_S, \, r_T,\, \alpha,\,\beta\right)$ is a matched product system of solutions if and only if they hold
	
	{\small
		\begin{center}
			\begin{minipage}[b]{.5\textwidth}
				\vspace{-\baselineskip}
				\begin{align*}
					\alpha_u\alpha_{\lambda_u^{-1}(v)}&=\alpha_v
\alpha_{\lambda_v^{-1}(u)}
				\end{align*}
			\end{minipage}%
			\hfill
			\begin{minipage}[b]{.5\textwidth}
				\vspace{-\baselineskip}
				\begin{align*}\beta_a\beta_{\lambda_a^{-1}(b)}=\beta_b
\beta_{\lambda_b^{-1}(a)} 
				\end{align*}
			\end{minipage}
		\end{center}
		\begin{center}
			\begin{minipage}[b]{.5\textwidth}
				\vspace{-\baselineskip}
				\begin{align*}\lambda_a\alpha_{\overline{u}}&=\alpha_u\lambda_{\overline{a}}
				\end{align*}
			\end{minipage}%
			\hfill
			\begin{minipage}[b]{.5\textwidth}
				\vspace{-\baselineskip}
				\begin{align*}\lambda_u\beta_{\overline{a}}=\beta_a\lambda_{\overline{u}}
				\end{align*}
			\end{minipage}
		\end{center}}
	\noindent for all $a, b \in S$ and $u,v \in T$.
 Note that one can easily choose $\alpha_u = \lambda_u$, for every $u\in T$, and $\beta_a = \lambda_a$, for every $a\in S$. Hence, we can specialize the \cref{teo_matched} for reflections of involutive left non-degenerate solutions that are $\lambda$-centralizing.

\begin{cor}\label{cor:mathc-inv-lnd}
    Let $(S,r_S)$ and $(T,r_T)$ be involutive left non-degenerate solutions and let $\left(r_{S}, r_{T}, \alpha, \beta\right)$ be a matched product system of solutions. Let us assume that $\kappa:S\to S$ and $\omega:T\to T$ are $\lambda$-centralizing reflections for $r_S$ and $r_T$, respectively, satisfying \eqref{cond_match_ref}.
Then, $\kappa \times \omega$ is a $\lambda$-centralizing reflection for $\left( S \bowtie T, r_{S}\bowtie
		r_{T}\right)$.
  \begin{proof}
      It is straightforward.
  \end{proof}
\end{cor}

\begin{rem}
    Note that, under the same assumptions of \cref{cor:mathc-inv-lnd}, if we choose $\alpha_u = \lambda_u$, for every  $u\in T$, and $\beta_a = \lambda_a$, for every  $a\in S$, then condition \eqref{cond_match_ref} is automatically satisfied.
\end{rem}

Notice that if $\left(r_{S}, r_{T}, \alpha, \beta\right)$ is a matched product system of solutions that are involutive right non-degenerate and $\kappa$ and $\omega$ are $\rho$-invariant reflections for $r_S$ and $r_T$, respectively, satisfying \eqref{cond_match_ref}, the reflection $\kappa \times \omega$ for the solution $r_S \bowtie r_T$ may not have the same behaviour. It is sufficient to consider the semidirect product of two solutions, as we show in the next example.

\begin{ex}
    Consider the trivial solutions $(S, r_S)$ and $(T, r_T)$ given by $r_S(a,b)=(b, a)$ and $r_T(u,v)=(v,u)$, for all $a,b \in S$ and $u, v \in T$, $\kappa=\id_S$, and let $\omega: T \to T$ be an arbitrary map.  Clearly, $\omega \in \K(T, r_T)$ is $\rho$-invariant.
    Moreover, let $\beta_a=\id_T$, for every $a \in S$, and let $\alpha_u:S \to S$ be maps such that $\alpha_u\alpha_v=\alpha_v\alpha_u$, for all $u,v \in T$. Then, $\left(r_{S}, r_{T}, \alpha, \beta\right)$ is a matched product system of solutions. Besides, by \cref{teo_matched}, $\kappa \times \omega$ is a reflection for $r_S \bowtie r_T$ and 
    \begin{align*}
      \kappa \times \omega\, \ \text{is $\rho$-invariant} \ \iff \ \alpha_u= \alpha_{\omega(u)}.
    \end{align*}
\end{ex}

\medskip

\subsection{The reflection of the strong semilattice of solutions}

Let us start by recalling the construction of strong semilattice of solutions provided in \cite[Theorem 4.1]{CCoSt20x-2}. Let $Y$ be a (lower) semilattice, 
	$\left\{\X{\alpha}\ \left|\ \alpha\in Y \right.\right\}$ a family of disjoint sets indexed by $Y$, and $\left(X_\alpha, r_{\alpha}\right)$ a solution, for every $\alpha\in Y$. 
    For all $\alpha, \beta \in Y$, we simply denote the product $\alpha\wedge\beta$ in $Y$ by the juxtaposition $\alpha\beta$.
	Moreover, for all $\alpha,\beta\in Y$ such that $\alpha\geq \beta$, let $\phii{\alpha}{\beta}:\X{\alpha}\to \X{\beta}$ be a map such that
	\begin{enumerate}
		\item $\phii{\alpha}{\alpha}$ is the identity map of $\X{\alpha}$, for every $\alpha\in Y$,
		\item $\phii{\beta}{\gamma}\phii{\alpha}{\beta} = \phii{\alpha}{\gamma}$, for all $\alpha, \beta, \gamma \in Y$ such that $\alpha \geq \beta \geq \gamma$,
		\item 
		$\left(\phii{\alpha}{\beta}\times \phii{\alpha}{\beta}\right)r_{\alpha}
		= r_{\beta}\left(\phii{\alpha}{\beta}\times \phii{\alpha}{\beta}\right)$, for all $\alpha, \beta\in Y$ such that  $\alpha \geq \beta$.
	\end{enumerate}
	Set $X = \bigcup\left\{\X{\alpha}\ \left|\ \alpha\in Y\right.\right\}$,  then the map $r:X\times X \longrightarrow X\times X$ defined by 
	\begin{align*}
	r\left(x, y\right)= 
	r_{\alpha\beta}\left(\phii{\alpha}{\alpha\beta}\left(x\right),
	\phii{\beta}{\alpha\beta}\left(y\right) \right),
	\end{align*}	
	for all $x\in \X{\alpha}$ and $y\in \X{\beta}$, is a solution on $X$. 
	We call the pair $\left(X, r\right)$ the \emph{strong semilattice $Y$ of solutions $\left(X_\alpha, r_\alpha\right)$}.\\
In general, a strong semilattice of solutions, with $|Y| > 1$, is not bijective and is degenerate, even in the case we involve bijective non-degenerate solutions, cf. \cite[p. 13]{CCoSt20x-2}.

 \smallskip
	
\begin{theor}
   	Let  $\left(X, r\right)$ be a strong semilattice $Y$ of solutions $\left(X_\alpha, r_\alpha\right)$ and assume that $\kappa_{\alpha}$ is a reflection for $\left(X_\alpha, r_{\alpha}\right)$, for every $\alpha \in Y$. Then, the map $\kappa: X \to X$ given by $\kappa(a)=\kappa_{\alpha}(a)$, for every $a \in X_{\alpha}$, is a reflection of the solution $(X,r)$ if and only if 
   	\begin{align}\label{cond_strong}\tag{S}
   	    \phii{\beta}{\alpha\beta}\kappa_{\beta}=\kappa_{\alpha\beta}\phii{\beta}{\alpha\beta}
   	\end{align}
   	is satisfied, for all $\alpha, \beta \in Y$ such that $\alpha \geq \beta$.
   	\begin{proof}
   	Let $a \in X_\alpha$ and $b \in X_\beta$. Then, we get
   	 \begin{align*}
     t(a,\kappa(b))&=\lambda_{\lambda_{\phii{\alpha}{\alpha\beta}\left(a\right)} \phii{\beta}{\alpha\beta}\kappa_{\beta}\left(b\right)}\kappa_{\alpha\beta}\rho_{ \phii{\beta}{\alpha\beta}\kappa_{\beta}\left(b\right)}\phii{\alpha}{\alpha\beta}\left(a\right)\\
&=\lambda_{\lambda_{\phii{\alpha}{\alpha\beta}\left(a\right)} \kappa_{\alpha\beta}\phii{\beta}{\alpha\beta}\left(b\right)}\kappa_{\alpha\beta}\rho_{ \kappa_{\alpha\beta}\phii{\beta}{\alpha\beta}\left(b\right)}\phii{\alpha}{\alpha\beta}\left(a\right) &\mbox{by \eqref{cond_strong}}\\
     &=t\left(\phii{\alpha}{\alpha\beta}\left(a\right), \kappa_{\alpha\beta}\phii{\beta}{\alpha\beta}\left(b\right) \right)\\
     &=t\left((\phii{\alpha}{\alpha\beta}\left(a\right),\phii{\beta}{\alpha\beta}\left(b\right)  \right)&\mbox{$\kappa_{\alpha\beta}$ is a reflection}\\
     &=\lambda_{\lambda_{\phii{\alpha}{\alpha\beta}\left(a\right)}\phii{\beta}{\alpha\beta}\left(b\right)}\kappa_{\alpha\beta}\rho_{\phii{\beta}{\alpha\beta}\left(b\right)}\phii{\alpha}{\alpha\beta}\left(a\right)\\
    &=t(a,b)
 \end{align*}
and
  \begin{align*}
      q(a, \kappa(b))&=\rho_{\kappa_{\alpha\beta}\rho_{\phii{\beta}{\alpha\beta}\kappa_{\beta}\left(b\right)}{\phii{\alpha}{\alpha\beta}\left(a\right)}}\lambda_{\phii{\alpha}{\alpha\beta}\left(a\right)}\phii{\beta}{\alpha\beta}\kappa_{\beta}\left(b\right)\\
      &=\rho_{\kappa_{\alpha\beta}\rho_{\kappa_{\alpha\beta}\phii{\beta}{\alpha\beta}\left(b\right)}{\phii{\alpha}{\alpha\beta}\left(a\right)}}\lambda_{\phii{\alpha}{\alpha\beta}\left(a\right)}\kappa_{\alpha\beta}\phii{\beta}{\alpha\beta}\left(b\right)&\mbox{by \eqref{cond_strong}}\\
      &=q\left(\phii{\alpha}{\alpha\beta}\left(a\right), \kappa_{\alpha\beta} \phii{\beta}{\alpha\beta}\left(b\right)\right)\\
      &=\kappa_{\alpha\beta} \left( q \left(\phii{\alpha}{\alpha\beta}\left(a\right), \phii{\beta}{\alpha\beta}\left(b\right) \right)\right)&\mbox{$\kappa_{\alpha\beta}$ is a reflection}\\
      &=\kappa_{\alpha\beta}\rho_{\kappa_{\alpha\beta}\rho_{\phii{\beta}{\alpha\beta}\left(b\right)}\phii{\alpha}{\alpha\beta}\left(a\right)}\lambda_{\phii{\alpha}{\alpha\beta}\left(a\right)}\phii{\beta}{\alpha\beta}\left(b\right)\\
      &=\kappa(q(a,b)),
  \end{align*}
  hence $\kappa$ is a reflection of the solution $r$.
   	\end{proof}
\end{theor}

\smallskip

\begin{cor}
      	Let  $\left(X, r\right)$ be a strong semilattice $Y$ of solutions $\left(X_\alpha, r_\alpha\right)$ and assume that  $\kappa_{\alpha} \in \K\left(X_\alpha, r_{\alpha}\right)$, for every $\alpha \in Y$, such that $\eqref{cond_strong}$ is satisfied. Consider the reflection $\kappa: X \to X$ given by $\kappa(a)=\kappa_{\alpha}(a)$, for every $a \in X_{\alpha}$. Then, the following hold:
      \begin{enumerate}
          \item if $\kappa_{\alpha}$ is $\lambda^{[\alpha]}$-centralizing for the solution $\left(X_\alpha, r_{\alpha}\right)$, for every $\alpha \in Y$, then $\kappa$ is $\lambda$-centralizing for $(X,r)$;
        \item if the map  $\kappa_{\alpha}$ is $\rho^{[\alpha]}$-invariant, for $\alpha \in Y$, then $\kappa$ is $\rho$-invariant.
      \end{enumerate} 
      \begin{proof}
          It is straightforward.
      \end{proof}
\end{cor}

\medskip

\subsection{The reflection of the twisted unions}

In \cite[Section 3.4]{ESS99}, Etingof, Schedler, and Soloviev provided a technique to construct involutive non-degenerate solutions in the finite case, called the \emph{twisted union}. A computer program generated all solutions, up to isomorphism, on sets of cardinality $|X| \leq 8$, and showed that almost all are twisted union of involutive non-degenerate solutions on smaller sets. We extend such a construction to arbitrary solutions.

\begin{prop}\label{prop:gen-twisted-union}
Let $(X, u)$ and $(Y, v)$ be two solutions. Set $Z = X \cupdot Y$ and consider the map $r:Z\times Z \to Z \times Z$ defined by
$$
    r\left(a,\,b\right)
    =\begin{cases}
    \ u\left(a,\,b\right) &\ \text{if}\quad a,\, b \in X\\
    \ v\left(a,\,b\right) &\ \text{if}\quad a,\, b \in Y\\
    \ \left(g\left(b\right),\,f\left(a\right)\right) &\  \text{if} \quad a \in X, b \in Y\\
    \ \left(\alpha\left(b\right),\,\beta\left(a\right)\right) &\ \text{if} \quad a \in Y, b \in X
\end{cases}$$
for all $a, b \in Z$, with $f,\alpha: X \to X$ and $g,\beta: Y \to Y$ maps. Then, $(Z,\,r)$ is a solution if and only if for all $x \in X$ and $y \in Y$ the following hold:
	{
		\begin{center}
			\begin{minipage}[b]{.5\textwidth}
				\vspace{-\baselineskip}
				\begin{align*}
					\alpha f = f\alpha
				\end{align*}
			\end{minipage}%
			\hfill\hfill\hfill
			\begin{minipage}[b]{.5\textwidth}
				\vspace{-\baselineskip}
				\begin{align*}g\beta=\beta g
				\end{align*}
			\end{minipage}
			\begin{minipage}[b]{.5\textwidth}
				\vspace{-\baselineskip}
				\begin{align*}
\left(f\times f\right)u = u\left(f\times f\right)
				\end{align*}
			\end{minipage}%
			\hfill\hfill
			\begin{minipage}[b]{.5\textwidth}
				\vspace{-\baselineskip}
				\begin{align*}\left(g\times g\right)v= v\left(g\times g\right)
				\end{align*}
			\end{minipage}
			\begin{minipage}[b]{.5\textwidth}
				\vspace{-\baselineskip}
				\begin{align*}\left(\alpha\times \alpha\right)u = u\left(\alpha\times \alpha\right)
				\end{align*}
			\end{minipage}%
			\hfill\hfill
			\begin{minipage}[b]{.5\textwidth}
				\vspace{-\baselineskip}
				\begin{align*}\left(\beta\times \beta\right)v = v\left(\beta\times \beta\right)
				\end{align*}
			\end{minipage}
			\begin{minipage}[b]{.5\textwidth}
				\vspace{-\baselineskip}
				\begin{align*}\alpha\lambda_{f\left(x\right)} = \lambda_x\alpha
				\end{align*}
			\end{minipage}%
			\hfill\hfill
			\begin{minipage}[b]{.5\textwidth}
				\vspace{-\baselineskip}
				\begin{align*}  \beta\rho_{g\left(y\right)} = \rho_y \beta
				\end{align*}
			\end{minipage}
			\begin{minipage}[b]{.5\textwidth}
				\vspace{-\baselineskip}
				\begin{align*}  f\rho_{\alpha(x)}= \rho_x f
				\end{align*}
			\end{minipage}%
			\hfill\hfill
			\begin{minipage}[b]{.5\textwidth}
				\vspace{-\baselineskip}
				\begin{align*}
                g\lambda_{\beta\left(y\right)} = \lambda_{y}g
				\end{align*}
			\end{minipage}
		\end{center}
	}
\noindent The pair $(Z,r)$ is called the \emph{twisted union of $(X, u)$ and $(Y, v)$ via $f,\alpha$ and $g,\beta$}.
\end{prop}
\begin{proof}
To get the claim it is enough to verify that 
\begin{align*}
    \eqref{first}&\Leftrightarrow 
         \alpha\lambda_{x} = \lambda_{\alpha\left(x\right)}\alpha, \ \  
         g\lambda_y = \lambda_{g\left(y\right)}g, \ \
         \alpha\lambda_{f\left(x\right)} = \lambda_{x}\alpha, 
         \ \text{and }\  g\lambda_{\beta\left(y\right)} = \lambda_{y}g, \\
    \eqref{second}&\Leftrightarrow
        f\alpha = \alpha f,\ \  
        g\beta = \beta g, \ \ 
        f\lambda_{x} = \lambda_{f\left(x\right)}f, \ \
        \beta\lambda_{y} = \lambda_{\beta\left(y\right)}\beta, \ \\
        &\,\quad \alpha\rho_{x} = \rho_{\alpha\left(x\right)}\alpha, \  \ \text{and }\ 
        g\rho_y = \rho_{g\left(y\right)}g,\\
    \eqref{third}&\Leftrightarrow
        f\rho_x = \rho_{f\left(x\right)}f, \ \
        \beta\rho_{y} = \rho_{\beta\left(y\right)}\beta, \ \
        f\rho_{\alpha\left(x\right)} = \rho_xf, \   \text{and }\ 
        \beta\rho_{g\left(y\right)} = \rho_y\beta,    
\end{align*}
for all $x\in X$ and $y\in Y$, respectively. 
\end{proof}

\smallskip

\begin{ex}\label{ex_lyu}
Let $\left(X, u\right)$ and $\left(Y, v\right)$ be the solutions given by $u\left(x, y\right) = \left(f\left(y\right), \alpha\left(x\right)\right)$ 
and $v\left(x, y\right) = \left(g\left(y\right), \beta\left(x\right)\right)$ with $f,\alpha:X\to X$ maps such that $f\alpha = \alpha f$ and $g,\beta:Y\to Y$ maps such that $g\beta = \beta g$. Then, the twisted union $(Z, r)$ of $\left(X, u\right)$ and $\left(Y, v\right)$ via $f, \alpha$ and $g, \beta$ is solution that does not belong to the class of Lyubashenko's solution.

\end{ex}

\begin{cor}\label{cor:gen-twisted-union}
    Let $(X,u)$ and $(Y,v)$ be bijective non-degenerate solutions, $f,\alpha \in \Sym_X$ and $g,\beta \in \Sym_Y$.
    Then, the twisted union $(Z, r)$ of $(X,u)$ and $(Y,v)$ via $f,\alpha$ and $g,\beta$ is a bijective non-degenerate solution. 
\end{cor}

\begin{rem}
    It is clear that \cref{cor:gen-twisted-union} includes the construction of twisted union provided in \cite{ESS99} if we consider two involutive non-degenerate solutions and aim to construct solutions of the same type. 
    Indeed, this is the case with $\alpha = f^{-1}$ and $\beta = g^{-1}$.
\end{rem}
    
More generally, by choosing suitable maps $f, \alpha, g$, and $\beta$, two involutive non-degenerate solutions can give rise to a bijective non-degenerate solution not necessarily involutive, as we show in the next example.
    \begin{ex}
        Let $(X, u)$ and $(Y, v)$ be the twist maps on $X$ and $Y$, respectively. As observed in \cref{ex:fg}, all maps $\kappa:X\to X$ and $\omega:Y\to T$ are reflections for $(X, u)$ and $(Y, v)$, respectively. Hence, by choosing $f = \alpha = \kappa$ and  $g = \beta = \omega$, conditions in \cref{prop:gen-twisted-union} are satisfied. Furthermore, if $\kappa$ or $\omega$ are not involutive, then $r$ is not involutive.
    \end{ex}
\medskip

%

In the following, we show that, in the construction of twisted union, one can choose as maps $f$, $\alpha$, $g$, and $\beta$, the reflections of the starting solutions if these are of derived type.
\begin{ex}\label{ex_twi}
    Let $\left(X, u\right)$ be the solution associated to a left rack $\left(X, \triangleright\right)$ and $(Y, v)$ the solution associated with a right rack $\left(Y, \triangleleft\right)$. Moreover, consider 
    $\kappa\in C_{\End(X,\triangleright)}\left(\LMlt\left(X,\triangleright\right)\right)$ (hence $\kappa$  is a reflection for $(X, u)$ by \cref{centraliz}) and $\omega=R_y$ a reflection for $(Y, v)$, with $y \in Y$.  If we set $f = \alpha = \kappa$, $g = \beta = \omega$, and $Z = X \cupdot Y$ then, 
    by \cref{moltipld_rifl_shelf} and
    \cref{prop:gen-twisted-union}, 
    the map $r:Z\times Z \to Z \times Z$ defined by
\begin{align*}
    r\left(a, b\right)
    =\begin{cases}
    \ \left(b, b\triangleright a\right)&\ \text{if}\quad a,\,b \in X\\
    \ \left(b\triangleleft a, a\right) &\ \text{if}\quad a,\, b\in Y\\
    \ \left(\omega\left(b\right),\,\kappa\left(a\right)\right) &\  \text{if} \quad a \in X, b \in Y\\
    \ \left(\kappa\left(b\right),\,\omega\left(a\right)\right) &\ \text{if} \quad a \in Y, b \in X
\end{cases}
\end{align*}
is a non-degenerate bijective solution.
\end{ex}

\smallskip

\begin{theor}\label{teo_twi}
    Let $(Z, r)$ be the twisted union of two solutions $(X, u)$ and $(Y, v)$, consider $\kappa \in \K(X, u)$ and $\kappa' \in \K(Y, v)$. Then, the map $\omega : Z \to Z$ defined by
    \begin{equation}\label{omega}
        \omega(a)=\begin{cases}
  \kappa(a) &\text{if}\quad a\in X\\
  \kappa'(a) &\text{if}\quad a \in Y
\end{cases}
    \end{equation}
is a reflection for $(Z,\,r)$ if and only if $\kappa (f\alpha)=(f \alpha) \kappa$ and $\kappa' (g\beta)=(g \beta) \kappa'$.
\begin{proof}
    It is straightforward.
\end{proof}
\end{theor}
\smallskip

Below, we provide some classes of examples obtained by \cref{teo_twi}.


\begin{ex}
    Let $\left(X, u\right)$ and $\left(Y, v\right)$ be equivalent solutions via a bijection $\varphi:X\to Y$. Let us denote by $\left(Z, r\right)$ the solution obtained through \cref{prop:gen-twisted-union} by considering $\alpha = f$ and $\beta = g$ maps such that $\left(X, u\right)$ and $\left(Y, v\right)$ are equivalent via $\alpha$ and $\beta$ with themselves, respectively.
    Furthermore, let $\kappa\in \mathcal{K}\left(X,u\right)$ and consider $\kappa' = \kappa_\varphi(=\varphi\kappa\varphi^{-1})$ that, by \cref{iso:refl}-$1.$, is a reflection for $\left(Y, v\right)$.
    Then, applying \cref{teo_twi}, we deduce that the map $\omega:Z\to Z$ defined as in \eqref{omega} is a reflection for $(Z,\,r)$.
\end{ex}

\begin{ex}
  Let us consider the twisted union $(Z, r)$ of $\left(X, u\right)$ and $\left(Y, v\right)$ in \cref{ex_lyu}. According to \cref{ex:fg}, any maps $\kappa: X \to X$ and $\kappa': Y \to Y$ such that $\kappa(f\alpha)=(f\alpha)k$ and $\kappa'(g\beta)=(g\beta)\kappa'$ are reflections for $(X, u)$ and $(Y, v)$, respectively. Thus, by \cref{teo_twi}, the map $\omega: Z \to Z$ defined as in \eqref{omega} is a reflection for $(Z, r).$ 
\end{ex}

\begin{ex}
 Let $(Z, r)$ be the twisted union of $(X, u)$ and $(Y, v)$ in \cref{ex_twi}, with $f=\alpha=\kappa \in \mathcal{K}(X, u)$ and $g=\beta=\kappa'\in \mathcal{K}(Y,v)$. Then by \cref{teo_twi} the map $\omega : Z \to Z$ defined as in \eqref{omega} is easily a reflection for $(Z,\,r)$.
\end{ex}

\newpage

\setcounter{section}{0}
\renewcommand{\thesection}{\Alph{section}}
\section{Reflections for solutions associated with skew braces}\label{sec:appendix}

We gather numerical data computing all reflections for some bijective non-degenerate solutions.
To this aim, we focus on the solutions associated with the well-studied algebraic structure of \emph{skew braces}. For more details on these structures, we refer the reader to \cite{GuVe17}. Reflections maps coming from skew braces merit further investigation which will be undertaken in future works.
In particular, we implement the \texttt{GAP} software and devise an elementary algorithmic approach to compute all reflections for the solutions associated to the first skew braces in the database furnished by the \texttt{YangBaxter} package \cite{YangBaxter}. 

The algorithm proceeds as explained below:

\begin{enumerate}
    \item[Input:] A positive integer $n$, a list of positive integers $T$;
    \item For each $m \in T$, consider the skew brace $B$ of order $n$ and type $m$ in the database, compute the associated solution and determine all its reflections with \cref{lemma_rifl_qt};
    \item Determine which reflections are $\lambda$-centralizing or $\rho$-invariant;
    \item[Output:] A list of the reflections for each type labelled as $\lambda$-centralizing or $\rho$-invariant.
\end{enumerate}

\medskip

The following table records the number of \emph{all} reflections for solutions associated with skew braces until the order $8$. 
Specifically, for each case, we enumerate the reflections which are \textbf{only} $\lambda$-centralizing, \textbf{only} $\rho$-invariant or \textbf{both}.
Lastly, for a better analysis, we also highlight the trivial and almost-trivial skew braces writing the type of the former in boldface and underlining the type of the latter.

{\small
\begin{table}[!htp]
\begin{center}
\begin{tabular}{|c|c|c|c|c|c|c|c|}
  \hline
    $\#$ & Type & $(B,+)$ & $(B,\circ)$ & Tot. & $\lambda$-centr. & $\rho$-inv. & Both  \\ \hline
    2 & \underline{\textbf{1}} & $C_2$ & $C_2$ & 4 & 0 & 0 & 4  \\\hline
    3 & \underline{\textbf{1}} & $C_3$ & $C_3$ & 27 & 0 & 0 & 27  \\ \hline
    \multirow{4}{1pt}{4} & \underline{\textbf{1}} & $C_4$ & $C_4$ & 256 & 0 & 0 & 256 \\ 
         & 2 & $C_4$ & $C_2\times C_2$ & 24 & 8 & 8 & 8  \\ 
         & \underline{\textbf{3}} & $C_2\times C_2$ & $C_2\times C_2$ & 256 & 0 & 0 & 256  \\ 
         &4 & $C_2\times C_2$ & $C_4$ & 24 & 8 & 8 & 8  \\\hline
    5 & \underline{\textbf{1}} & $C_5$ & $C_5$ & 3125 & 0 & 0 & 3125  \\ \hline
    \multirow{6}{1pt}{6} & \textbf{1} & $S_3$ & $S_3$ & 49 & 48 & 0 & 1 \\
         & \underline{2} & $S_3$ & $S_3$ & 2 & 0 & 0 & 2  \\
         & 3 & $S_3$ & $C_6$ & 11 & 8 & 0 & 1  \\
         & 4 & $S_3$ & $C_6$ & 28 & 27 & 0 & 1  \\
         & 5 & $C_6$ & $S_3$ & 864 & 135 & 720 & 9 \\
         & \underline{\textbf{6}} & $C_6$ & $C_6$ & 46656 & 0 & 0 & 46656 \\ \hline
     7 & \underline{\textbf{1}} & $C_7$ & $C_7$ & 823543 & 0 & 0 & 823543  \\ \hline
     \end{tabular}
     \caption{$\lambda$-centr. and $\rho$-inv. reflections for solutions associated with skew braces until order $7$.}
\end{center}
\end{table}
}

\small{
\begin{table}[!htp]
    \begin{center}
    \begin{tabular}{|c|c|c|c|c|c|c|}
    \hline
    Type & $(B,+)$ & $(B,\circ)$ & Tot. & $\lambda$-centr. & $\rho$-inv. & Both \\ \hline
        1 & $C_8$ & $C_8$ & 77824 & 12288 & 61440 & 4096 \\ 
        2 & $C_8$ & $D_8$ & 67328 & 1792 & 65280 & 256 \\ 
        3 & $C_8$ & $Q_8$ & 67328 & 1792 & 65280 & 256 \\ 
        \underline{\textbf{4}} & $C_8$ & $C_8$ & 16777216 & 0 & 0 & 16777216 \\ 
        5 & $C_8$ & $C_4\times C_2$ & 4352 & 112 & 240 & 16 \\ 
        6 & $C_4\times C_2$ & $C_2\times C_2 \times C_2$ & 77824 & 12288 & 61440 & 4096 \\
        7 & $C_4\times C_2$ & $C_4\times C_2$ & 12480 & 224 & 224 & 32 \\ 
        8 & $C_4\times C_2$ & $C_2\times C_2 \times C_2$ & 4800 & 224 & 224 & 32 \\ 
        \underline{\textbf{9}} & $C_4\times C_2$ & $C_4\times C_2$ & 16777216 & 0 & 0 & 16777216 \\ 
        10 & $C_4\times C_2$ & $C_4\times C_2$ & 77824 & 12288 & 61440 & 4096 \\ 
        11 & $C_4\times C_2$ & $C_4\times C_2$ & 4352 & 112 & 240 & 16 \\ 
        12 & $C_4\times C_2$ & $D_8$ & 81664 & 16128 & 65280 & 256 \\ 
        13 & $C_4\times C_2$ & $D_8$ & 81664 & 16128 & 65280 & 256 \\ 
        14 & $C_4\times C_2$ & $C_4\times C_2$ & 77824 & 12288 & 61440 & 4096 \\ 
        15 & $C_4\times C_2$ & $Q_8$ & 81664 & 16128 & 65280 & 256 \\ 
        16 & $C_4\times C_2$ & $C_4\times C_2$ & 4800 & 224 & 224 & 32 \\ 
        17 & $C_4\times C_2$ & $D_8$ & 960 & 224 & 224 & 32 \\ 
        18 & $C_4\times C_2$ & $D_8$ & 1152 & 63 & 0 & 1 \\ 
        19 & $C_4\times C_2$ & $D_8$ & 5312 & 2016 & 224 & 32 \\ 
        20 & $D_8$ & $Q_8$ & 1184 & 1024 & 0 & 32 \\ 
        \textbf{21} & $D_8$ & $D_8$ & 3040 & 3008 & 0 & 32 \\ 
        \underline{22} & $D_8$ & $D_8$ & 64 & 0 & 0 & 64 \\ 
        23 & $D_8$ & $Q_8$ & 192 & 32 & 0 & 32 \\ 
        24 & $D_8$ & $C_8$ & 288 & 256 & 16 & 16 \\ 
        25 & $D_8$ & $C_4\times C_2$ & 1120 & 1024 & 0 & 32 \\ 
        26 & $D_8$ & $C_4\times C_2$ & 1120 & 1024 & 0 & 32 \\ 
        27 & $D_8$ & $C_8$ & 128 & 32 & 16 & 16 \\ 
        28 & $D_8$ & $C_2\times C_2 \times C_2$ & 1120 & 1024 & 0 & 32 \\  
        29 & $D_8$ & $D_8$ & 192 & 32 & 0 & 32 \\ 
        30 & $D_8$ & $C_4\times C_2$ & 288 & 32 & 0 & 32 \\ 
        31 & $D_8$ & $D_8$ & 1184 & 1024 & 0 & 32 \\ 
        \textbf{32} & $Q_8$ & $Q_8$ & 3040 & 3008 & 0 & 32 \\ 
        \underline{33} & $Q_8$ & $Q_8$ & 64 & 0 & 0 & 64 \\ 
        34 & $Q_8$ & $C_2\times C_2 \times C_2$ & 288 & 32 & 0 & 32 \\ 
        35 & $D_8$ & $Q_8$ & 192 & 32 & 0 & 32 \\ 
        36 & $Q_8$ & $D_8$ & 1184 & 1024 & 0 & 32 \\ 
        37 & $Q_8$ & $C_4\times C_2$ & 1120 & 1024 & 0 & 32 \\
        38 & $Q_8$ & $C_8$ & 128 & 32 & 16 & 16 \\ 
        39 & $Q_8$ & $C_8$ & 288 & 256 & 16 & 16 \\ 
        \underline{\textbf{40}} & $C_2\times C_2 \times C_2$ & $C_2\times C_2 \times C_2$ & 16777216 & 0 & 0 & 16777216 \\ 
        41 & $C_2\times C_2 \times C_2$ & $C_2\times C_2 \times C_2$ & 12480 & 224 & 224 & 32 \\ 
        42 & $C_2\times C_2 \times C_2$ & $Q_8$ & 960 & 224 & 224 & 32 \\ 
        43 & $C_2\times C_2 \times C_2$ & $C_4\times C_2$ & 4800 & 224 & 224 & 32 \\ 
        44 & $C_2\times C_2 \times C_2$ & $C_4\times C_2$ & 77824 & 12288 & 61440 & 4096 \\ 
        45 & $C_2\times C_2 \times C_2$ & $C_4\times C_2$ & 4352 & 112 & 240 & 16 \\ 
        46 & $C_2\times C_2 \times C_2$ & $D_8$ & 81664 & 16128 & 65280 & 256 \\ 
        47 & $C_2\times C_2 \times C_2$ & $D_8$ & 1278 & 63 & 0 & 1 \\ 
        \hline
    \end{tabular}
\end{center}
\caption{$\lambda$-centr. and $\rho$-inv. reflections for solutions associated with skew braces of order $8$.}
\end{table}
}

\normalsize{ 
Finally, for the sake of completeness, the following table collects all reflections for solutions associated with some skew braces which are \textbf{only} $\lambda$-invariant or \textbf{only} $\rho$-centralizing. Additionally, in the last column, we enumerate the reflections that belong to all \textit{four} classes. 


\medskip

{\small
\begin{table}[!htp]
\begin{center}
\begin{tabular}{|c|c|c|c|c|c|c|}
  \hline
    $\#$ & Type & $(B,+)$ & $(B,\circ)$ & $\lambda$-inv. & $\rho$-centr. & All \\ \hline
    2 & \underline{\textbf{1}} & $C_2$ & $C_2$ & 0 & 0 & 4 \\\hline
    3 & \underline{\textbf{1}} & $C_3$ & $C_3$ & 0 & 0 & 27 \\ \hline
    \multirow{4}{1pt}{4} & \underline{\textbf{1}} & $C_4$ & $C_4$ & 0 & 0 & 256 \\ 
         & 2 & $C_4$ & $C_2 \times C_2$ & 0 & 0 & 8 \\ 
         & \underline{\textbf{3}} & $C_2 \times C_2$ & $C_2 \times C_2$ & 0 & 0 & 256 \\ 
         &4 & $C_2 \times C_2$ & $C_4$ & 0 & 0 & 8 \\\hline
    5 & \underline{\textbf{1}} & $C_5$ & $C_5$ & 0 & 0 & 3125 \\ \hline
    \multirow{6}{1pt}{6} & \textbf{1} & $S_3$ & $S_3$ & 0 & 0 & 1 \\
         & \underline{2} & $S_3$ & $S_3$ & 0 & 0 & 1 \\
         & 3 & $S_3$ & $C_6$ & 2 & 0 & 1 \\
         & 4 & $S_3$ & $C_6$ & 0 & 0 & 1 \\
         & 5 & $C_6$ & $S_3$ & 0 & 0 & 9 \\
         & \underline{\textbf{6}} & $C_6$ & $C_6$ & 0 & 0 & 46656 \\ \hline
     7 & \underline{\textbf{1}} & $C_7$ & $C_7$ & 0 & 0 & 823543 \\ \hline
     \end{tabular}
     \caption{$\lambda$-inv. and $\rho$-centr. reflections for solutions associated with skew braces until order $7$.}
\end{center}
\end{table}
}

\section*{Data availability}

The GAP code used for the computation is available from the authors on request.

\section*{Acknowledgments}
The authors are all members of GNSAGA (INdAM). This work was supported by the Dipartimento di Matematica e Fisica ``Ennio De Giorgi" - Universit\`{a} del Salento. 

A.~Albano was supported by a scholarship financed by the Ministerial Decree no. 118/2023, based on the NRRP - funded by the European Union - NextGenerationEU - Mission 4. 

M.~Mazzotta was supported by the ``HaMMon (Hazard Mapping and vulnerability Monitoring) - Spoke 2" project. 

\medskip
We thank the referee for carefully reading our manuscript and the remarks and suggestions that helped us improve the paper.

\bigskip
\bibliography{bibliography}

\bigskip
\bigskip

\begin{itemize}
\item[] Andrea ALBANO (ORCID: 0009-0007-5397-5505)\\
 \emph{Email address:} 
\texttt{andrea.albano@unisalento.it} 
 \vspace{3mm}
    \item[]  Marzia MAZZOTTA (ORCID: 0000-0001-6179-9862) \\
    \emph{Email address:} \texttt{marzia.mazzotta@unisalento.it}
 \vspace{3mm}
       \item[] Paola STEFANELLI  (ORCID: 0000-0003-3899-3151) \\
        \emph{Email address:} \texttt{paola.stefanelli@unisalento.it}
\end{itemize}

\end{document}